\documentclass[arxiv=true]{agn_article}

\usepackage{agn_all}
\usepackage{framed}
\usepackage{tabularx}
\usepackage{caption}
\usepackage{lscape}
\usepackage{rotating}
\usepackage{float}
\usepackage{wrapfig}

\newcommand{\citet}[2][]{\citeauthor{#2} \cite[#1]{#2}}
\newcommand{\Wiso}{W_{\mathrm{iso}}}
\newcommand{\Wvol}{W_{\mathrm{vol}}}
\renewcommand{\dd}{\displaystyle}
\renewcommand{\ddx}{\frac{\mathrm{d}}{\mathrm{d}x}}
\renewcommand{\ddy}{\frac{\mathrm{d}}{\mathrm{d}y}}
\renewcommand{\ddt}{\frac{\mathrm{d}}{\mathrm{d}t}}
\renewcommand{\ddz}{\frac{\mathrm{d}}{\mathrm{d}z}}
\newcommand{\lmax}{\lambda_{\mathrm{max}}}
\newcommand{\lmin}{\lambda_{\mathrm{min}}}
\newcommand{\Wliniso}{W_\mathrm{iso}^\mathrm{lin}}
\newcommand{\Wlinvol}{W_\mathrm{vol}^\mathrm{lin}}
\renewcommand{\Wlin}{W^\mathrm{lin}}
\newcommand{\sigmaliniso}{\sigma_\mathrm{iso}^\mathrm{lin}}
\newcommand{\sigmalinvol}{\sigma_\mathrm{vol}^\mathrm{lin}}
\newcommand{\sigmalin}{\sigma^\mathrm{lin}}
\DeclareMathOperator{\arcosh}{arcosh}
\newcommand{\WSL}{W_{\mathrm{SL}}}

\begin{document}
\title{\vspace{-2cm}Sharp rank-one convexity conditions in planar isotropic elasticity for the additive volumetric-isochoric split}
\date{\today}
\author{%
	Jendrik Voss\thanks{%
		Corresponding author: Jendrik Voss,\quad Lehrstuhl f\"{u}r Nichtlineare Analysis und Modellierung, Fakult\"{a}t f\"{u}r Mathematik, Universit\"{a}t Duisburg-Essen, Thea-Leymann Str. 9, 45127 Essen, Germany; email: max.voss@uni-due.de%
	}\quad and \quad
	Ionel-Dumitrel Ghiba\thanks{%
	Ionel-Dumitrel Ghiba, Alexandru Ioan Cuza University of Ia\c si, Department of Mathematics, Blvd. Carol I, no. 11, 700506 Ia\c si, Romania; Octav Mayer Institute of Mathematics of the Romanian Academy, Ia\c si Branch, 700505 Ia\c si; and Lehrstuhl f\"{u}r Nichtlineare Analysis und Modellierung, Fakult\"{a}t f\"{u}r Mathematik, Universit\"{a}t Duisburg-Essen, Thea-Leymann Str. 9, 45127 Essen, Germany, email: dumitrel.ghiba@uni-due.de, dumitrel.ghiba@uaic.ro%
	}\quad and\quad
	Robert J.\ Martin\thanks{%
		Robert J.\ Martin,\quad Lehrstuhl f\"{u}r Nichtlineare Analysis und Modellierung, Fakult\"{a}t f\"{u}r Mathematik, Universit\"{a}t Duisburg-Essen, Thea-Leymann Str. 9, 45127 Essen, Germany; email: robert.martin@uni-due.de%
	}\\ and\quad%
	Patrizio Neff\thanks{%
		Patrizio Neff,\quad Head of Lehrstuhl f\"{u}r Nichtlineare Analysis und Modellierung, Fakult\"{a}t f\"{u}r	Mathematik, Universit\"{a}t Duisburg-Essen, Thea-Leymann Str. 9, 45127 Essen, Germany, email: patrizio.neff@uni-due.de%
		}
}
\maketitle
\vspace*{-0.5cm}
\begin{abstract}
\noindent
	We consider the volumetric-isochoric split in planar isotropic hyperelasticity and give a precise analysis of rank-one convexity criteria for this case, showing that the Legendre-Hadamard ellipticity condition separates and simplifies in a suitable sense. Starting from the classical two-dimensional criterion by Knowles and Sternberg, we can reduce the conditions for rank-one convexity to a family of one-dimensional coupled differential inequalities. In particular, this allows us to derive a simple rank-one convexity classification for generalized Hadamard energies of the type $W(F)=\frac{\mu}{2}\.\frac{\norm{F}^2}{\det F}+f(\det F)$; such an energy is rank-one convex if and only if the function $f$ is convex.
\end{abstract}
\\[.65em]
{\parskip=-0.5mm \tableofcontents}
\section{Introduction}
Undoubtedly, one of the most important constitutive requirements in nonlinear elasticity is the rank-one convexity of the elastic energy $W\col\GLp(n)\to\R$. For $C^2$-smooth energy functions, rank-one convexity is nothing else than the Legendre-Hadamard ellipticity condition expressing the ellipticity of the Euler-Lagrange equations
\begin{align}
	\Div DW(\nabla\varphi)=0\qquad\text{associated to the variational problem}\qquad\int_\Omega W(\nabla\varphi)\.\dx\to\min.
\end{align}
Thus, rank-one convexity for smooth energies amounts to the requirement
\begin{align}
	D_F^2W(F).(\xi\otimes\eta,\xi\otimes\eta)\geq c\,\abs{\xi}^2\.\abs{\eta}^2,\qquad c\geq 0\qquad\text{for all }\;\xi,\eta\in\R^n\setminus\{0\}\label{eq:legendreHadamard}
\end{align}
for all $F\in\GLpn$. This condition is equivalent to the requirement that the acoustic tensor $Q(F,\eta)\in\Sym(n)$ with $Q_{ik}(F,\eta)=\sum_{j,l=1}^n\frac{\partial^2W(F)}{\partial  F_{ij}\partial F_{kl}}\cdot\eta_j\eta_l$ is positive semi-definite for all directions $\eta\in\R^n$, where
\begin{align}
	\iprod{\xi,Q(F,\eta).\xi}_{\R^{n\times n}}=D_F^2W(F).(\xi\otimes\eta,\xi\otimes\eta).
\end{align}
For the dynamic problem, a non-negative acoustic tensor means that the system is hyperbolic \cite{gavrilyuk2016example,ndanou2014criterion}. When the dynamic equation is linearized at $F\in\GLp(n)$, then there exists a traveling wave solution with real wave speeds if and only if the problem is Legendre-Hadamard elliptic at $F$ \cite{silhavy1997mechanics}. Moreover, the static response is metastable (stable against infinitesimal perturbations) under Legendre-Hadamard ellipticity (if $c>0$ in \eqref{eq:legendreHadamard}). The loss of ellipticity is generally related to instability phenomena (separation into arbitrary fine phase mixtures \cite{silhavy1997mechanics,grabovsky2018rank}, shear banding) and possible discontinuous equilibrium solutions \cite{rosakis1990ellipticity}.

Checking rank-one convexity for a given nonlinear elastic material can be quite challenging, see e.g.\ \cite{agn_martin2018non,bertram2007rank,pedregal2018does,agn_neff2014loss} and Appendix \ref{appendix:dumitrel}, although John Ball's polyconvexity \cite{ball1976convexity} as an easy-to-verify sufficient condition can often be helpful \cite{agn_schroder2010poly,agn_martin2018polyconvex}. A certain simplification occurs in the isotropic case. Indeed, any isotropic energy on $\GLp(n)$ admits a representation in terms of the singular values $W(F)=g(\lambda_1,\cdots,\lambda_n)$, where $g\col(0,\infty)^n\to\R$ is permutation invariant. In planar isotropic hyperelasticity, Knowles and Sternberg have provided a criterion for rank-one convexity in terms of $g$, conclusively reducing the problem to the singular values representation \cite{knowles1976failure,knowles1978failure,knowles1975ellipticity,silhavy2002n,davies1991simple}, see also \cite{dacorogna01,de2012note,simpson1983copositive,zubov1992effective,parry1995planar,vsilhavy2003so}. In the Knowles-Sternberg result, a family of inequality constraints in terms of first and second derivatives of $g$ (including the Baker-Ericksen inequality \cite{bakerEri54}) has to be checked, which can still be daunting in practice. 

The ellipticity condition in the planar and isotropic case has also been reformulated in terms of yet different representations, see e.g.\ \cite{hamburger1987characterization,de2012note,rosakis1994relation}, always in the form of a set of differential inequality constraints involving the used special representation. Generalizations to the case of functions defined not on $\GLp(2)$ but on $\R^{2\times 2}$ have been obtained by Aubert \cite{aubert1995necessary} and Hamburger \cite{hamburger1987characterization}. For some special families of energy functions more detailed information is available. For example, in the fluid-like case \cite{Dacorogna08}
\begin{align}
	W(F)=f(\det F) \qquad\text{with }\; f\col(0,\infty)\to\R\,,\qquad\text{$W$ is rank-one convex if and only if $f$ is convex.} 
\end{align}
Moreover, for the Hadamard-material \cite[Theorem 5.58ii)]{Dacorogna08,ball1984w1} (see also \cite{dantas2006equivalence})
\begin{align}
	&W\col\GLp(n)\to\R,\qquad W(F)=\norm{F}^\alpha+f(\det F)\qquad \text{with }\;1\leq\alpha<2n\,,\\
	&W \text{ polyconvex}\quad\iff\qquad W\text{ quasiconvex}\quad\iff\quad W\text{ rank-one convex}\quad\iff\quad f\text{ convex}.\notag
\end{align}
In a recent contribution \cite{agn_martin2015rank,agn_martin2019envelope}, we have been able to classify all planar, isotropic rank-one convex functions $W\col\GLp(2)\to\R$ which are isochoric (conformally invariant), i.e.\ satisfy 
\begin{align}
	 W(A\.F\.B)=W(F)\qquad\text{for all}\quad A,B\in\left\{a\.R\in\GLp(2)\;|\;a\in(0,\infty)\,,\;R\in\SO(2)\right\}.
\end{align}
For these energies we have the simple characterization\footnote{The isotropy of the energy implies $h(t)=h\big(\frac{1}{t}\big)$ for all $t\in(0,\infty)$, cf.\ equation \eqref{eq:definitionOfh}.}
\begin{align}
	W(F)=h\bigg(\frac{\lambda_1}{\lambda_2}\bigg)\,,\quad h\col(0,\infty)\to\R\,,\qquad\text{then $W$ is rank-one convex if and only if $h$ is convex,} 
\end{align}
where $\lambda_1,\lambda_2>0$ are the singular values of $F$. Similarly, rank-one convexity in isotropic planar incompressible elasticity $W\col\SL(2)\to\R$ can be easily checked \cite{agn_ghiba2018rank}: for an energy of the form
\begin{align}
\label{eq:incompressibleEnergyRepresentation}
	W(F)=\phi\left(\lmax-\frac{1}{\lmax}\right) \qquad\text{with }\;\phi\col[0,\infty)\to\R\,,\qquad\begin{array}{l}
		\text{$W$ is rank-one convex if and only if}\\\text{{$\phi$ is convex and nondecreasing.} }
	\end{array}
\end{align}
These results also allow for an explicit calculation of the quasiconvex relaxation for conformally invariant and incompressible isotropic planar hyperelasticity \cite{agn_martin2019envelope,agn_martin2019quasiconvex}.
%
%
%
%

\subsection{The volumetric-isochoric split}
In isotropic linear elasticity, the quadratic elastic energy $\Wlin$ and the linear Cauchy stress tensor $\sigmalin$ can be uniquely represented as
\begin{align}
	\Wlin(\nabla u)&=\mu\.\norm{\sym\nabla u}^2+\frac{\lambda}{2}\.(\tr\nabla u)^2=\mu\.\norm{\dev_n\sym\nabla u}^2+\frac{\kappa}{2}\.(\tr\nabla u)^2\,,\notag\\
	\sigmalin&=D\Wlin(\nabla u)=2\mu\.\dev_n\sym\nabla u+\kappa\.\tr(\nabla u)\.\id\,,\label{eq:linEla}
\end{align}
where $\mu,\lambda $ are the Lam\'e-constants, $\kappa$ is the bulk-modulus, $\nabla u$ denotes the displacement gradient and $\dev_n X=X-\frac{1}{n}\tr(X)\cdot\id$ is the deviatoric part of $X\in\Rnn$, cf.\ Appendix \ref{appendix:lin}. The right hand side of the energy is automatically additively separated into pure infinitesimal shape change and infinitesimal volume change, respectively, with a similar additive split of the Cauchy stress tensor into a deviatoric part and a spherical part, depending only on the shape change and volumetric change, respectively.

A natural finite strain extension of \eqref{eq:linEla} is given by the additive volumetric-isochoric split \cite{horgan2009volumetric}, i.e.\ by assuming the energy potential $W$ to have the form
\begin{align}
	W(F)=\Wiso\bigg(\frac{F}{\sqrt[n]{\det F}}\bigg)+\Wvol(\det F)\,,
\end{align}
where $\Wiso$ is isochoric (conformally invariant), since $\Wiso(s\.F)=\Wiso\big(\frac{s\.F}{\sqrt[n]{\det(s\.F)}}\big)=\Wiso\big(\frac{F}{\sqrt[n]{\det F}}\big)$ for all $s>0$. Here, again, the energy contribution is additively split into the isochoric part taking only the shape change into account (depending only on $\frac{F}{\sqrt[n]{\det F}}$), and a part penalizing only the volumetric extension in $\det F$. Note the multiplicative decomposition
\begin{align}
	F=\frac{F}{\sqrt[n]{\det F}}\cdot\sqrt[n]{\det F}\cdot\id\,,\qquad\det\left(\frac{F}{\sqrt[n]{\det F}}\right)=1
\end{align}
of the deformation gradient $F$ itself into volume preserving and shape changing part \cite{richter1948,agn_graban2019richter,flory1961thermodynamic,sidoroff1974modele,simo1988framework,ogden1978nearly,federico2010volumetric,murphy2018modelling}. Such type of energy functions are often used when modeling slightly compressible material behavior 
\cite{Ciarlet1988,agn_hartmann2003polyconvexity,favrie2014thermodynamically,ogden1978nearly,charrier1988existence} or when otherwise no detailed information on the actual response of the material is available \cite{simo1988framework,gavrilyuk2016example,ndanou2017piston}. In the nonlinear regime, this split is not a fundamental law of nature for isotropic bodies (as in the linear case) but rather introduces a convenient form of the energy.
Formally, this split can also be generalized to the anisotropic case, where it shows, however, some severe deficiencies \cite{federico2010volumetric,murphy2018modelling} from a modeling point of view.\footnote{For example, a perfect sphere made of an anisotropic material and subject only to all around pressure would remain spherical for volumetric-isochorically decoupled energies.} 

It has been first shown by Richter \cite{richter1948,agn_graban2019richter,richter1949verzerrung,agn_neff2020axiomatic} (see also Flory \cite{flory1961thermodynamic} and Sansour \cite{sansour2008physical}) that the accompanying symmetric nonlinear Kirchhoff stress tensor $\tau=\det F\cdot\sigma$ (but not the Cauchy stress as repeatedly claimed in \cite{ndanou2014criterion,ndanou2017piston}) admits a similar additive structure in the sense that
\begin{align}
	\tau(F)=\dev\tau+\frac{1}{n}\.\tr(\tau)\cdot\id=\tau_{\rm iso}+\tau_{\rm vol}\.\id\,,\qquad\tau_{\rm vol}=\det F\cdot\Wvol'(\det F)\label{eq:kirchhoff}
\end{align}
in which the deviatoric part of the Kirchhoff stress only depends on $\Wiso$ and the spherical part only depends on $\Wvol$. A typical example of the foregoing volumetric-isochoric format is the geometrically nonlinear quadratic Hencky energy \cite{agn_neff2015geometry,Hencky1928,sendova2005strong,agn_martin2018polyconvex} 
\begin{align}
	W_{\rm H}(F)&=\mu\.\norm{\dev_n\log V}^2+\frac{\kappa}{2}\left(\tr\log V\right)^2,\qquad V=\sqrt{F\.F^T}\label{eq:hencky}\\
	\tau_{\rm H}&=2\mu\.\dev_n\log V+\kappa\.\tr(\log V)\cdot\id=2\mu\.\log\frac{V}{\sqrt[3]{\det V}}+\kappa\.\log\det F\cdot\id\notag
\end{align}
as well as its physically nonlinear extension, the exponentiated Hencky-model \cite{agn_neff2015exponentiatedI,agn_neff2014exponentiatedIII}
\begin{align}
	W_{\rm eH}&=\frac{\mu}{k}\.e^{k\.\norm{\dev_n\log V}^2}+\frac{\kappa}{2\khat}\.e^{\khat\left(\tr\log V\right)^2}\,,\label{eq:expHencky}\\
	\tau_{\rm eH}&=2\mu\.e^{k\.\|\dev_n\log\,V\|^2}\cdot\dev_n\log V+\underbrace{\kappa\.e^{\khat\left(\tr\log V\right)^2}\.\tr(\log V)\cdot\id}_{=\kappa\.e^{\khat\left(\log\det F\right)^2}\.\log\det F\cdot\id}\,,\notag
\end{align}
which has been used for the stable computation of the inversion of tubes \cite{agn_nedjar2018finite}, the modeling of tire-derived material \cite{agn_montella2015exponentiated} or applications in bio-mechanics \cite{agn_schroder2018exponentiated}. While it is well known that \eqref{eq:hencky} and \eqref{eq:expHencky} are overall not rank-one convex \cite{agn_neff2015exponentiatedI}, there does not exist any elastic energy depending on $\norm{\dev_n\log V}^2$, $n\geq 3$ that is rank-one convex \cite{agn_ghiba2015ellipticity}. The situation is surprisingly completely different for the planar case. Indeed, \eqref{eq:expHencky} for $n=2$ is not only rank-one convex, but also polyconvex \cite{agn_neff2015exponentiatedII,agn_ghiba2015exponentiated} if $k\geq\frac{1}{8}$.

Because of this counter-intuitive result when descending from three-dimensions to two-dimensions, we became interested in the precise qualitative nature of the volumetric-isochoric split in the planar case with respect to rank-one convexity.\footnote{The three-dimensional case of volumetric-isochoric split energies on $\GLp(3)$ has previously been considered in \cite{ndanou2014criterion}, where the authors propose involved sufficient criteria for rank-one convexity.} In the planar isotropic case (exclusively), the isochoric energy part $\Wiso$ admits a number of equivalent representations, e.g.
\begin{align*}
	W(F)=\Wiso\bigg(\frac{F}{\sqrt[n]{\det F}}\bigg)+\Wvol(\det F)=\widehat\psi\big(\norm{\dev_2\log U}^2\big)+\fhat\big(\tr(\log U)\big)\,.
\end{align*}
In particular, for planar additively volumetric-isochoric split energies we can uniquely write without loss of generality \cite[Lemma 3.1]{agn_martin2015rank}
\begin{align}
	W(F)=h\bigg(\frac{\lambda_1}{\lambda_2}\bigg)+f(\lambda_1\lambda_2)\,,\qquad h,f\col(0,\infty)\to\R\,,\quad h(t)=h\bigg(\frac{1}{t}\bigg)\qquad\text{for all }\;t\in(0,\infty)\,,\label{eq:definitionOfh}
\end{align}
where $\lambda_1,\lambda_2>0$ denote the singular values of $F$ and $h,f$ are given functions. It is now tempting to believe that the rank-one convexity conditions on $\widehat\psi$ and $\fhat$, respectively on $\Wiso$ and $\Wvol$ also allow for a sort of split (like for the Kirchhoff stress tensor in \eqref{eq:kirchhoff}) and can be reduced to separate statements on $h$ and $f$. However, this is not even the case in planar linear elasticity, where $\Wlin$ from \eqref{eq:linEla} is rank-one convex in the displacement gradient $\nabla u$ if and only if $\mu\geq 0$ and $\mu+\kappa\geq 0$, see Appendix \ref{appendix:lin}. This means that rank-one convexity of $\Wlin$ implies that $\Wliniso$ is rank-one convex, whereas $\Wlinvol$ might not be rank-one convex. We therefore need to expect some coupling in the conditions for $h$ and $f$. 

Our main result in this respect (Theorem \ref{theorem:main}) can be summarized as follows: If $W$ is altogether rank-one convex then, either $\Wiso$ or $\Wvol$ must be rank-one convex (i.e.\ $h$ or $f$ are convex), but contrary to isotropic linear elasticity, $\Wiso$ may also be non rank-one convex (see the example in Section \ref{sec:examples}). Although the conditions for $h$ and $f$ do not decouple completely, it is possible to reduce rank-one convexity to a family of coupled one-dimensional conditions.
 
The paper is now structured as follows. After a short introduction and visualization of rank-one convexity on $\GLp(2)$, we start with criteria by Knowles and Sternberg for the arbitrary objective-isotropic case on the two-dimensional set of the singular values $\lambda_1,\lambda_2>0$ of $F$. Lemma \ref{lemma:voliso} gives the corresponding conditions for the volumetric-isochoric split. The main Theorem \ref{theorem:main} shows that it is indeed possible to reduce the conditions to a coupled family of one-dimensional differential inequalities which makes testing for rank-one convexity much more convenient in the volumetric-isochoric split case. Apart from a list of additional necessary conditions, we derive a simple classification for generalized Hadamard energies and give two examples of non-trivial rank-one convex energies for which $\Wiso$ and $\Wvol$, respectively, are not rank-one convex. Finally, we show how criteria for the invertibility of the Cauchy stress response are also simplified in the volumetric-isochoric split case and highlight the connection of this invertibility property to rank-one convexity.
%
%
%
%
\section[\boldmath Rank-one convexity conditions on $\GLp(2)$]{\boldmath Rank-one convexity conditions on $\GLp(2)$}

We first recall the following basic definitions.

\begin{definition}
	An energy function $W\col\Rnn\to\R$ is called \emph{rank-one convex} if the mapping
	\begin{align}
		t\mapsto W(F+t\.\xi\otimes\eta)\qquad\text{is convex in }t\text{ for all } F\in\R^{n\times n}\text{ and }\xi,\eta\in\R^n\,.
	\end{align}
\end{definition}
\begin{definition}
	An energy function $W\in C^2(\Rnn;\R)$ is called \emph{Legendre-Hadamard elliptic} if
	\begin{align}
		D^2 W(F).[\xi\otimes\eta,\xi\otimes\eta]\geq 0\qquad\text{for all } F\in\R^{n\times n}\text{ and }\xi,\eta\in\R^n
	\end{align}
	which is equivalent to rank-one convexity.
\end{definition}
Throughout the following, we will use the identification $W(F)=g(\lambda_1,\lambda_2)$ for any isotropic planar energy function $W\col\GLp(2)\to\R$, where $\lambda_1,\lambda_2>0$ are the singular values of $F$, i.e.\ the eigenvalues of $\sqrt{F^TF}$, and the function $g\col(0,\infty)^2\to\R$ is permutation invariant.
%
%
%
%
\subsection{Ellipticity domains for some nonlinear energy functions}
For an isotropic energy function $W\col\GLp(2)\to\R$, it is possible to visualize the ellipticity domain in terms of the singular values $\lambda_1,\lambda_2>0$ of $F$, as shown in Figure \ref{fig:notRankOneconvex}. The isotropy of $W$ ensures the symmetry at the bisector, i.e.\ $W(\lambda_1,\lambda_2)=W(\lambda_2,\lambda_1)$. For purely isochoric energies $W(F)=\Wiso\big(\frac{F}{\sqrt{\det F}}\big)=h\big(\frac{\lambda_1}{\lambda_2}\big)$, the ellipticity domain consists only of cones and can therefore be reduced to a one-dimensional problem at $\det F=1$ or any other fixed determinant. Of course, this is not the case for energies with an arbitrary volumetric part $f(\lambda_1\cdot\lambda_2)$.
\begin{figure}[h!]
	\begin{minipage}[t]{.49\linewidth}
  		\centering
	 	\includegraphics[width=\textwidth]{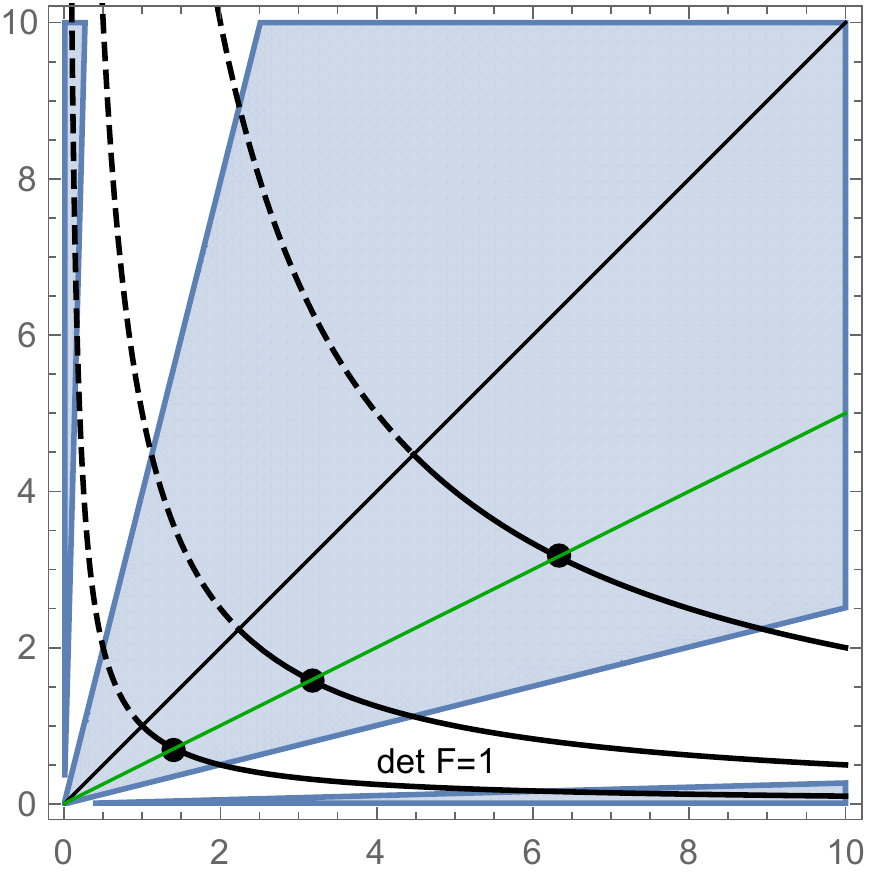}
	 \end{minipage}%
  	 \hfill%
	 \begin{minipage}[t]{.49\linewidth}
	   \centering
	   \includegraphics[width=\textwidth]{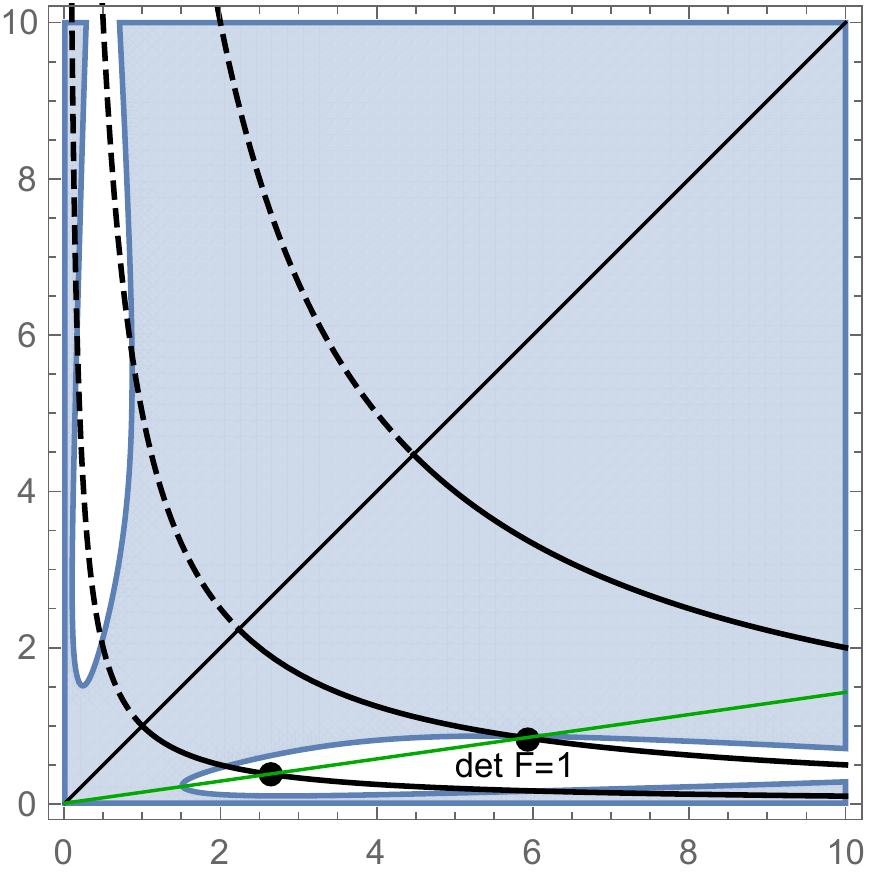}
	\end{minipage}
  \caption{\label{fig:notRankOneconvex} The ellipticity domain in terms of the singular values $\lambda_1,\lambda_2>0$ of the exponentiated Hencky-model with $k<\frac{1}{8}$. \ Left: Only the non rank-one convex isochoric part $h(t)=e^{\frac{1}{10}\.(\log t)^2}$. Rank-one convexity at one point on the curve $\det F=\lambda_1\lambda_2=c$ (black) for arbitrary $c>0$ implies rank-one convexity at the straight line (green) connecting the origin to this point. \ Right: For additive coupling with the volumetric part $f(z)=\frac{1}{1000}\left(z-\frac{1}{z}\right)^2$, the above invariance is lost for non-isochoric energies.}
\end{figure}
%
%
%
%
\subsection{The classical Knowles and Sternberg ellipticity criterion}
For rank-one convexity on $\GLp(2)$ Knowles and Sternberg, cf.\ \cite[p.~318]{silhavy1997mechanics}) established the following important and useful criterion.
\begin{lemma}[{Knowles and Sternberg, cf.\ \cite[p.~308]{silhavy1997mechanics} see also \cite{parry1995planar,dacorogna01,vsilhavy2002convexity,vsilhavy1999isotropic,aubert1987faible,aubert1995necessary}}]
\label{lemma:knowlesSternberg}
	Let $W\col\GLp(2)\to\R$ be an objective-isotropic function of class $C^2$ with the representation in terms of the singular values of the deformation gradient $F$ given by $W(F)=g(\lambda_1,\lambda_2)$, where $g\in C^2((0,\infty)^2,\R)$. Then $W$ is rank-one convex if and only if the following five conditions are satisfied:
	\begin{itemize}
		\item[i)] $\dd\underbrace{g_{xx}\geq 0\quad\text{and}\quad g_{yy}\geq 0}_{\text{\normalfont\enquote{TE-inequalities} (separate convexity)}}\qquad \text{for all}\quad x,y\in (0,\infty)$\,,
		\item[ii)] \qquad$\dd\underbrace{\frac{x\.g_x-y\.g_y}{x-y}\geq 0}_{\text{\normalfont\enquote{BE-inequalities}}}\qquad\text{for all}\quad x,y\in (0,\infty)\,,\; x\neq y$\,,
		\item[iii)] \qquad$\dd g_{xx}-g_{xy}+\frac{g_x}{x}\geq 0\quad\text{and}\quad g_{yy}-g_{xy}+\frac{g_y}{y}\geq 0\qquad\text{for all}\quad x,y\in (0,\infty)\,,\; x=y$\,,
		\item[iv)]\qquad$\dd\sqrt{g_{xx}\,g_{yy}}+g_{xy}+\frac{g_x-g_y}{x-y}\geq 0\qquad\text{for all}\quad x,y\in (0,\infty)\,,\; x\neq y$\,,
		\item[v)]\qquad$\dd\sqrt{g_{xx}\,g_{yy}}-g_{xy}+\frac{g_x+g_y}{x+y}\geq 0\qquad\text{for all}\quad x,y\in (0,\infty)$\,.
	\end{itemize}
	Furthermore, if all the above inequalities are strict, then $W$ is strictly rank-one convex.	
\end{lemma}
This criterion has recently been generalized to plane strain orthotropic materials \cite{aguiar2019strong}, see also \cite{walton2003sufficient,agn_merodio2006note}. Hamburger \cite{hamburger1987characterization} gives an extension to $\R^{2\times 2}$ instead of $\GLp(2)$ in the same format. In this paper, we restrict ourselves to the class of two-dimensional isotropic energy functions \mbox{$W\col\GLp(2)\to\R$} with additive volumetric-isochoric split
\begin{align}
	W(F)=\Wiso\bigg(\frac{F}{\sqrt{\det F}}\bigg)+\Wvol(\det F)=h\bigg(\frac{\lambda_1}{\lambda_2}\bigg)+f(\lambda_1\lambda_2)\label{eq:VolIso}
\end{align}
as introduced in \eqref{eq:definitionOfh}, where $\lambda_1,\lambda_2\in(0,\infty)$ are the singular values of $F$ and $h\col(0,\infty)\to\R$ satisfies $h(t)=h\big(\frac{1}{t}\big)$ for all $t\in(0,\infty)$.\footnote{%
	The latter requirement is due to the invariance of the singular value representation $g(x,y)=h\big(\frac{x}{y}\big)+f(x\cdot y)$ under permutation of the arguments, which implies $h\big(\frac{x}{y}\big)=h\big(\frac{x}{y}\big)$ for all $x,y>0$.
}
Note that the case $W(F)=+\infty$ is excluded by definition.

Starting with a rank-one convex energy function $W$ with additive volumetric-isochoric split, it is not clear whether the separate parts $\Wiso$ and $\Wvol$ need to be rank-one convex, too. Indeed, subsequently we will see that this is not always the case.
%
%
%
%
\subsection{Necessary and sufficient conditions for the planar volumetric-isochoric split}
The following Lemma expresses the Knowles-Sternberg criterion (Lemma \ref{lemma:knowlesSternberg}) in terms of the functions $h$ and $f$ for the specific form $g(x,y)=h\big(\frac{x}{y}\big)+f(x\cdot y)$ of $g$ in the case of an additive volumetric-isochoric split.
\begin{lemma}\label{lemma:voliso}
	Let $W\col\GLp(2)\to\R$ be an objective-isotropic function of class $C^2$ with additive volumetric-isochoric split, with the representation in terms of the singular values of the deformation gradient $F$ given by $W(F)=g(\lambda_1,\lambda_2)=h\big(\frac{\lambda_1}{\lambda_2}\big)+f(\lambda_1\lambda_2)$, where $f,h\in C^2((0,\infty),\R)$ and\.%
	\footnote{%
		The isotropy of the energy implies the symmetry under permutation of the two singular values $\lambda_1$ and $\lambda_2$, thus
	\[
		h(t)=h\bigg(\frac{\lambda_1}{\lambda_2}\bigg)=h\bigg(\frac{\lambda_2}{\lambda_1}\bigg)=h\bigg(\frac{1}{t}\bigg)\qquad\text{for all }t\in(0,\infty)\,.
	\]%
	}
	$h(t)=h\big(\frac{1}{t}\big)$ for all $t\in(0,\infty)$. Then $W$ is rank-one convex if and only if the following four conditions are satisfied:
	\begin{itemize}
		\item[A)] \quad$\dd t^2\.h''(t)+z^2\.f''(z)\geq 0\qquad \text{for all}\quad t,z\in(0,\infty)$\qquad (equivalent to Knowles and Sternberg i)\,)\,,
		\item[B)] \quad$\dd h'(t)\geq 0\qquad\text{for all}\quad t\geq 1$\qquad (equivalent to Knowles and Sternberg ii)\,)\,,
		\item[C)]\quad$\dd \frac{2t}{t-1}\.h'(t)-t^2\.h''(t)+z^2\.f''(z)\geq 0$\quad or\quad$a(t)+\left[b(t)-c(t)\right]z^2\.f''(z)\geq 0$\quad for all \;\mbox{$t,z\in(0,\infty)\,,\;t\neq 1$}\qquad (equivalent to Knowles and Sternberg iv)\,)\,,
		\item[D)]\quad$\dd \frac{2t}{t+1}\.h'(t)+t^2\.h''(t)-z^2\.f''(z)\geq 0\quad\text{or}\quad a(t)+\left[b(t)+c(t)\right]z^2\.f''(z)\geq 0\qquad\text{for all}\quad t,z\in(0,\infty)$\qquad (equivalent to Knowles and Sternberg v)\,)\,,
	\end{itemize}
	where
	\begin{align*}
		a(t)&= t^2(t^2-1)\.h'(t)h''(t)-2t\.h'(t)^2,\qquad b(t)= \left(t^2+3\right)h'(t)+2t\.(t^2+1)\.h''(t)\,,\\
		c(t)&=4t\left(h'(t)+t\.h''(t)\right).
	\end{align*}
\end{lemma}
\begin{remark}
	Note that Knowles and Sternberg condition iii) is redundant in this case, since it is already implied by condition B), cf.\ \cite{dacorogna01}.
\end{remark}
\begin{proof}
\parindent0em
\parskip0.5em
	We start by computing the partial derivatives
	\begin{align}
		g_x&=\ddx\left[h\bigg(\frac{x}{y}\bigg)+f(xy)\right]=\frac{1}{y}\.h'\bigg(\frac{x}{y}\bigg)+y\.f'(xy)\,,\notag\\
		g_y&=\ddy\left[h\bigg(\frac{x}{y}\bigg)+f(xy)\right]=-\frac{x}{y^2}\.h'\bigg(\frac{x}{y}\bigg)+x\.f'(xy)\,,\notag\\
		g_{xx}&=\ddx\left[\frac{1}{y}\.h'\bigg(\frac{x}{y}\bigg)+y\.f'(xy)\right]=\frac{1}{y^2}\.h''\bigg(\frac{x}{y}\bigg)+y^2\.f''(xy)\,,\\
		g_{xy}&=\ddx\left[-\frac{x}{y^2}\.h'\bigg(\frac{x}{y}\bigg)+x\.f'(xy)\right]=-\frac{1}{y^2}\.h'\bigg(\frac{x}{y}\bigg)-\frac{x}{y^3}\.h''\bigg(\frac{x}{y}\bigg)+f'(xy)+x\.y\.f''(xy)\,,\notag\\
		g_{yy}&=\ddy\left[-\frac{x}{y^2}\.h'\bigg(\frac{x}{y}\bigg)+x\.f'(xy)\right]=\frac{2x}{y^3}\.h'\bigg(\frac{x}{y}\bigg)+\frac{x^2}{y^4}\.h''\bigg(\frac{x}{y}\bigg)+x^2\.f''(xy)\notag
	\end{align}
	of the energy function with respect to the singular values. Next, we utilize the specific form of the additive split by introducing the transformation of variables
	\begin{align}
		t\colonequals \frac{x}{y}\,,\qquad z\colonequals x\.y\qquad\iff\qquad x=\sqrt{z\.t}\,,\qquad y=\sqrt{\frac{z}{t}}\,,\qquad t,z\in (0,\infty)
	\end{align}
	which allows us to decouple the conditions for $g(x,y)$ into conditions for $f(z)$ and $h(t)$. Moreover, the invariance property $h(t)=h\big(\frac{1}{t}\big)$ of the isochoric part $h$ yields
	\begin{align}
		h'(t)&=\ddt h(t)=\ddt h\left(\frac{1}{t}\right)=-\frac{1}{t^2}\.h'\left(\frac{1}{t}\right),\\
		h''(t)&=\ddt h'(t)=\ddt\left[-\frac{1}{t^2}\.h'\left(\frac{1}{t}\right)\right]=\frac{2}{t^3}\.h'\left(\frac{1}{t}\right)+\frac{1}{t^4}\.h''\left(\frac{1}{t}\right).\label{eq:symmetryofh}
	\end{align}
	Note that for $t=1$, equation \eqref{eq:symmetryofh} already implies\footnote{The equality $h'(1)=0$ even holds for all $h\in C^1((0,\infty),\R)$, since in that case,
	\begin{align*}
		h'(1)=\lim_{t\searrow 1}\frac{h(t)-h(1)}{t-1}=\lim_{t\searrow 1}\frac{h\big(\frac{1}{t}\big)-h\big(\frac{1}{1}\big)}{1-\frac{1}{t}}\.\frac{1}{t}\overset{s\colonequals\frac{1}{t}}{=}\lim_{s\nearrow 1}\frac{h(s)-h(1)}{1-s}\.1=-h'(1)\,.
	\end{align*}}
	\begin{align*}
		h''(1)=h'(1)+h''(1)\qquad\implies\qquad h'(1)=0\,.
	\end{align*}
	In the following, we transform the conditions i)-v) of Knowles and Sternberg Lemma \ref{lemma:knowlesSternberg} one by one.\par
	\textbf{Knowles and Sternberg i)}\quad $g_{xx}\geq 0$\quad and\quad$g_{yy}\geq 0$\quad for all $x,y,\in(0,\infty)$: We compute
	\begin{align}
		&&&g_{xx}=\frac{1}{y^2}\.h''\bigg(\frac{x}{y}\bigg)+y^2\.f''(xy)\geq 0\notag\\
		&\iff&&\frac{t}{z}\.h''(t)+\frac{z}{t}\.f''(z)\geq 0\qquad\overset{\cdot zt}{\iff}\qquad t^2\.h''(t)+z^2\.f''(z)\geq 0\,,\label{eq:conditionSC}\\
		&\text{and}&&g_{yy}=\frac{2x}{y^3}\.h'\bigg(\frac{x}{y}\bigg)+\frac{x^2}{y^4}\.h''\bigg(\frac{x}{y}\bigg)+x^2\.f''(xy)\geq 0\label{eq:conditionSC2}\\
		&\overset{\cdot y^2}{\iff}&&\frac{2x}{y}\.h'\bigg(\frac{x}{y}\bigg)+\frac{x^2}{y^2}\.h''\bigg(\frac{x}{y}\bigg)+x^2y^2\.f''(xy)\geq 0\quad\iff\quad 2t\.h'(t)+t^2\.h''(t)+z^2\.f''(z)\geq 0\,.\notag
	\end{align}
	The first condition \eqref{eq:conditionSC} already has the desired form. For the second condition \eqref{eq:conditionSC2} we use the symmetry property \eqref{eq:symmetryofh} for a further transformation, using the equality
	\begin{align}
		2t\.h'(t)+t^2\.h''(t)=2t\left[-\frac{1}{t^2}\.h'\left(\frac{1}{t}\right)\right]+t^2\left[\frac{2}{t^3}\.h'\left(\frac{1}{t}\right)+\frac{1}{t^4}\.h''\left(\frac{1}{t}\right)\right]=\frac{1}{t^2}\.h''\left(\frac{1}{t}\right)
	\end{align}
	to substitute
	\begin{align}
		&&&2t\.h'(t)+t^2\.h''(t)+z^2\.f''(z)=\frac{1}{t^2}\.h''\left(\frac{1}{t}\right)+z^2\.f''(z)\geq 0\qquad\text{for all}\quad t\in(0,\infty)\notag\\
		&\overset{s\colonequals\frac{1}{t}}{\iff}&&s^2\.h''(s)+z^2\.f''(z)\geq 0\qquad\text{for all}\quad z\in(0,\infty)\,,
	\end{align}
	which is equivalent to condition \eqref{eq:conditionSC}. Thus\par
	\fbox{\parbox{0.99\textwidth}{\textbf{i)}\hfill
		\centering $\dd g_{xx}\geq 0\text{ and }g_{yy}\geq 0\quad\text{for all }\;x,y\in(0,\infty)\;\;\iff\;\; t^2\.h''(t)+z^2\.f''(z)\geq 0\quad\text{for all }\;t,z\in(0,\infty)$.
	\hfill\textbf{A)}}}\par
	Note that A) implies $\inf_{t\in(0,\infty)}t^2\.h''(t)\,,\;\inf_{z\in(0,\infty)}z^2\.f''(z)\in\R$.

	\textbf{Knowles and Sternberg ii)}\quad $\dd\frac{x\.g_x-y\.g_y}{x-y}\geq 0$\quad for all $x,y\in (0,\infty)\,,\; x\neq y$: We compute
	\begin{align}
		&&&\hspace{-6cm}\frac{x\.g_x-y\.g_y}{x-y}=\frac{\frac{x}{y}\.h'\big(\frac{x}{y}\big)+xy\.f'(xy)-\left(-\frac{x}{y}\.h'\big(\frac{x}{y}\big)+xy\.f'(xy)\right)}{x-y}=\frac{2\.\frac{x}{y}\.h'\big(\frac{x}{y}\big)}{x-y}\geq 0\notag\\
		&\iff\qquad 2\frac{x}{y}\.h'\bigg(\frac{x}{y}\bigg)\geq 0\qquad\text{for all }x>y&&\text{and}\qquad 2\.\frac{x}{y}\.h'\bigg(\frac{x}{y}\bigg)\leq 0\qquad\text{for all }x<y\notag\\
		&\iff\qquad 2t\.h'(t)\geq 0\qquad\text{for all }t>1&&\text{and}\qquad 2t\.h'\left(t\right)\leq 0\qquad\text{for all }t<1\label{eq:conditionBE}\\
		&\overset{\cdot\frac{1}{2t}}{\iff}\qquad h'(t)\geq 0\qquad\text{for all }t>1&&\text{and}\qquad h'(t)\leq 0\qquad\text{for all }t<1\,.\notag
	\end{align}
	Since the invariance property \eqref{eq:symmetryofh}, i.e.\ $h(t)=h\big(\frac{1}{t}\big)$ for all $t\in\R$, yields
	\begin{align}
		h'(t)\geq 0\qquad\text{for all}\quad t>1\qquad\iff\qquad h'(t)\leq 0\qquad\text{for all}\quad t<1\,,
	\end{align}
	we obtain the equivalence\par
	\fbox{\parbox{0.99\textwidth}{\textbf{ii)}\hfill
		\centering $\dd\frac{x\.g_x-y\.g_y}{x-y}\geq 0\quad\text{for all }x,y\in(0,\infty)\qquad\iff\qquad h'(t)\geq 0\quad\text{for all }t\geq 1\,.$
		\hfill\textbf{B)}}}\par
	Note that B) implies $h'(1)=0$ as well as $h''(1)\geq 0$.\par
	\textbf{Knowles and Sternberg iii)}\quad $\dd g_{xx}-g_{xy}+\frac{g_x}{x}\geq 0\,, \quad g_{yy}-g_{xy}+\frac{g_y}{y}\geq 0$\quad for all $x=y\in (0,\infty)$:
	
	 We compute
	\begin{align}
		&&&g_{xx}(x,x)-g_{xy}(x,x)+\frac{g_x}{x}(x,x)=\frac{1}{x^2}\.h''(1)+x^2\.f''(x^2)+\frac{1}{x^2}\.h'(1)+\frac{1}{x^2}\.h''(1)-f'(x^2)\notag\\
		&&&\hspace{5.2cm}-x^2\.f''(x^2)+\frac{1}{x^2}\.h'(1)+f'(x^2)\geq 0\notag\\
		&\iff&& \frac{2}{x^2}\.h''(1)+\frac{2}{x^2}\.\underbrace{h'(1)}_{=0}\geq 0\qquad\iff\qquad h''(1)\geq 0\,,\\
		&\text{and}&&g_{yy}(x,x)-g_{xy}(x,x)+\frac{g_y}{y}(x,x)=\frac{2}{x^2}\.h'(1)+\frac{1}{x^2}\.h''(1)+x^2\.f''(x^2)+\frac{1}{x^2}\.h'(1)+\frac{1}{x^2}\.h''(1)\notag\\
		&&&\hspace{5.2cm}-f'(x^2)-x^2\.f''(x^2)-\frac{1}{x^2}\.h'(1)+f'(x^2)\geq 0\notag\\
		&\iff&& \frac{2}{x^2}\.h''(1)+\frac{2}{x^2}\.\underbrace{h'(1)}_{=0}\geq 0\qquad\iff\qquad h''(1)\geq 0\,.
	\end{align}
	Thus\par
	\fbox{\parbox{0.99\textwidth}{\textbf{iii)}\hfill
		\centering $\dd g_{xx}-g_{xy}+\frac{g_x}{x}\geq 0\,, \quad g_{yy}-g_{xy}+\frac{g_y}{y}\geq 0\quad\text{for all }x=y\in(0,\infty)\qquad\iff\qquad h''(1)\geq 0\,,$
		\hfill\phantom{.}}}\par
	the latter condition being already implied by condition B), which in turn follows from Knowles and Sternberg ii).\par
	\textbf{Knowles and Sternberg iv)+v)}\quad We compute
	\begin{align}
		g_{xx}\,g_{yy}&=\left(\frac{1}{y^2}\.h''\bigg(\frac{x}{y}\bigg)+y^2\.f''(xy)\right)\left(\frac{2x}{y^3}\.h'\bigg(\frac{x}{y}\bigg)+\frac{x^2}{y^4}\.h''\bigg(\frac{x}{y}\bigg)+x^2\.f''(xy)\right)\notag\\
		&=\frac{2x}{y^5}\.h'\bigg(\frac{x}{y}\bigg)h''\bigg(\frac{x}{y}\bigg)+\frac{x^2}{y^6}\.h''\bigg(\frac{x}{y}\bigg)^2+\left[\frac{2x}{y}\.h'\bigg(\frac{x}{y}\bigg)+\frac{x^2}{y^2}\.h''\bigg(\frac{x}{y}\bigg)+\frac{x^2}{y^2}\.h''\bigg(\frac{x}{y}\bigg)\right]f''(xy)\notag\\
		&\phantom{=}\;+x^2y^2\.f''(xy)^2\notag\\
		&=\frac{1}{z^2}\left[2t^3\.h'(t)h''(t)+t^4\.h''(t)^2\right]+\left[2t\.h'(t)+2t^2\.h''(t)\right]f''(z)+z^2f''(z)^2\,,
		\displaybreak[0]\\
		\frac{g_x-g_y}{x-y}&=\frac{\frac{1}{y}\.h'\bigg(\frac{x}{y}\bigg)+y\.f'(xy)+\frac{x}{y^2}\.h'\bigg(\frac{x}{y}\bigg)-x\.f'(xy)}{x-y}=\frac{1}{y^2}\frac{x+y}{x-y}\.h'\bigg(\frac{x}{y}\bigg)-f'(xy)\,, \notag
		\displaybreak[0]\\
		g_{xy}+\frac{g_x-g_y}{x-y}&=-\frac{1}{y^2}\.h'\bigg(\frac{x}{y}\bigg)-\frac{x}{y^3}\.h''\bigg(\frac{x}{y}\bigg)+f'(xy)+xy\.f''(xy)+\frac{1}{y^2}\frac{x+y}{x-y}\.h'\bigg(\frac{x}{y}\bigg)-f'(xy)\notag\\
		&=\frac{1}{y^2}\left(\frac{x+y}{x-y}-1\right)\.h'\bigg(\frac{x}{y}\bigg)-\frac{x}{y^3}\.h''\bigg(\frac{x}{y}\bigg)+xy\.f''(xy)\notag\\
		&=\frac{2}{y\.(x-y)}\.h'\bigg(\frac{x}{y}\bigg)-\frac{x}{y^3}\.h''\bigg(\frac{x}{y}\bigg)+xy\.f''(xy)\,,
		\displaybreak[0]\\
		\frac{g_x+g_y}{x+y}&=\frac{\frac{1}{y}\.h'\bigg(\frac{x}{y}\bigg)+y\.f'(xy)-\frac{x}{y^2}\.h'\bigg(\frac{x}{y}\bigg)+x\.f'(xy)}{x+y}=\frac{1}{y^2}\frac{y-x}{x+y}\.h'\bigg(\frac{x}{y}\bigg)+f'(xy)\,,\notag
		\displaybreak[0]\\
		-g_{xy}+\frac{g_x+g_y}{x+y}&=\frac{1}{y^2}\.h'\bigg(\frac{x}{y}\bigg)+\frac{x}{y^3}\.h''\bigg(\frac{x}{y}\bigg)-f'(xy)-xy\.f''(xy)+\frac{1}{y^2}\frac{y-x}{x+y}\.h'\bigg(\frac{x}{y}\bigg)+f'(xy)\notag\\
		&=\frac{1}{y^2}\left(\frac{x-y}{x+y}+1\right)\.h'\bigg(\frac{x}{y}\bigg)+\frac{x}{y^3}\.h''\bigg(\frac{x}{y}\bigg)-xy\.f''(xy)\notag\\
		&=\frac{2}{y\.(x+y)}\.h'\bigg(\frac{x}{y}\bigg)+\frac{x}{y^3}\.h''\bigg(\frac{x}{y}\bigg)-xy\.f''(xy)
	\end{align}
	and substitute
	\begin{align}
		y\.(x-y)&=\sqrt{\frac{z}{t}}\left(\sqrt{zt}-\sqrt{\frac{z}{t}}\right)=z-\frac{z}{t}=\frac{z\.(t-1)}{t}&&\implies\qquad\frac{2}{y\.(x-y)}=\frac{2t}{z\.(t-1)}\,,\\
		y\.(x+y)&=\sqrt{\frac{z}{t}}\left(\sqrt{zt}+\sqrt{\frac{z}{t}}\right)=z+\frac{z}{t}=\frac{z\.(t+1)}{t}&&\implies\qquad\frac{2}{y\.(x+y)}=\frac{2t}{z\.(t+1)}
	\end{align}
	to find
	\begin{align}
		g_{xy}+\frac{g_x-g_y}{x-y}&=\frac{2}{y\.(x-y)}\.h'\bigg(\frac{x}{y}\bigg)-\frac{x}{y^3}\.h''\bigg(\frac{x}{y}\bigg)+xy\.f''(xy)=\frac{1}{z}\left[\frac{2t}{t-1}\.h'(t)-t^2\.h''(t)\right]+z\.f''(z)\,,\\
		-g_{xy}+\frac{g_x+g_y}{x+y}&=\frac{2}{y\.(x+y)}\.h'\bigg(\frac{x}{y}\bigg)+\frac{x}{y^3}\.h''\bigg(\frac{x}{y}\bigg)-xy\.f''(xy)=\frac{1}{z}\left[\frac{2t}{t+1}\.h'(t)+t^2\.h''(t)\right]-z\.f''(z)\,.\notag
	\end{align}
	For further computations, we need to express the Knowles-Sternberg conditions in a form not involving square roots\footnote{The square root expression does not allow for a proper decoupling of $h(t)$ and $z(z)$, since
	\begin{align*}
		\sqrt{g_{xx}\,g_{yy}}&=\sqrt{\frac{1}{z^2}\left[2t^3\.h'(t)h''(t)+t^4\.h''(t)^2\right]+\left[2t\.h'(t)+2t^2\.h''(t)\right]f''(z)+z^2f''(z)^2}\\
		&=\frac{1}{z}\sqrt{t^2h'(t)^2-t^2h'(t)^2+2t^3\.h'(t)h''(t)+t^4\.h''(t)^2+2\left[t\.h'(t)+t^2\.h''(t)\right]z^2f''(z)+z^4f''(z)^2}\\
		&=\frac{1}{z}\sqrt{\left[t\.h'(t)+t^2\.h''(t)+z^2f''(z)\right]^2-t^2h'(t)^2}.
	\end{align*}} in the term $\sqrt{g_{xx}\,g_{yy}}$. However, we will need to pay close attention to the sign of occurring terms.\par
	\textbf{Knowles and Sternberg iv)}\quad $\dd\sqrt{g_{xx}\,g_{yy}}+g_{xy}+\frac{g_x-g_y}{x-y}\geq 0$\quad for all $x,y\in (0,\infty)\,,\; x\neq y$: It holds
	\begin{align}
		\sqrt{g_{xx}\,g_{yy}}\geq -g_{xy}-\frac{g_x-g_y}{x-y}\qquad\iff\qquad g_{xy}+\frac{g_x-g_y}{x-y}\geq 0\quad\text{or}\quad g_{xx}\,g_{yy}\geq \left(g_{xy}+\frac{g_x-g_y}{x-y}\right)^2.\label{eq:condition4cases}
	\end{align}
	For the first case in equation \eqref{eq:condition4cases} we compute
	\begin{align}
		g_{xy}+\frac{g_x-g_y}{x-y}=\frac{1}{z}\left[\frac{2t}{t-1}\.h'(t)-t^2\.h''(t)\right]+z\.f''(z)\geq 0\quad\overset{\cdot z}{\iff}\quad\frac{2t}{t-1}\.h'(t)-t^2\.h''(t)+z^2\.f''(z)\geq 0\,.
	\end{align}
	For the second case in equation \eqref{eq:condition4cases} we start with
	\begin{align}
		\left(g_{xy}+\frac{g_x-g_y}{x-y}\right)^2&=\left(\frac{1}{z}\left[\frac{2t}{t-1}\.h'(t)-t^2\.h''(t)\right]+z\.f''(z)\right)^2\notag\\
		&=\frac{1}{z^2}\left[\frac{4t^2}{(t-1)^2}\.h'(t)^2-\frac{4t^3}{t-1}\.h'(t)h''(t)+t^4\.h''(t)^2\right]\\
		&\phantom{=}\;+\left[\frac{4t}{t-1}\.h'(t)-2t^2\.h''(t)\right]f''(z)+z^2f''(z)^2\notag
	\end{align}
	and compute
	\begin{align}
		&&&g_{xx}\,g_{yy}\geq \left(g_{xy}+\frac{g_x-g_y}{x-y}\right)^2\notag\\
		&\iff&&\frac{1}{z^2}\left[2t^3\.h'(t)h''(t)+t^4\.h''(t)^2\right]+\left[2t\.h'(t)+2t^2\.h''(t)\right]f''(z)+z^2f''(z)^2\notag\\
		&&&\geq \frac{1}{z^2}\left[\frac{4t^2}{(t-1)^2}\.h'(t)^2-\frac{4t^3}{t-1}\.h'(t)h''(t)+t^4\.h''(t)^2\right]+\left[\frac{4t}{t-1}\.h'(t)-2t^2\.h''(t)\right]f''(z)+z^2f''(z)^2\notag\\
		&\overset{\mathclap{\cdot z^2}}{\iff}&&\left(2t^3+\frac{4t^3}{t-1}\right)h'(t)h''(t)-\frac{4t^2}{(t-1)^2}\.h'(t)^2+\left[\left(2t-\frac{4t}{t-1}\right)h'(t)+4t^2\.h''(t)\right]z^2\.f''(z)\geq 0\notag\\
		&\overset{\mathclap{\cdot\frac{(t-1)^2}{2t}}}{\iff}&& t^2(t^2-1)\.h'(t)h''(t)-2t\.h'(t)^2+\left[\left(t^2-4t+3\right)h'(t)+2t\.(t-1)^2\.h''(t)\right]z^2\.f''(z)\geq 0\,.
	\end{align}
	Thus, with the definitions
	\begin{align*}
		a(t)&= t^2(t^2-1)\.h'(t)h''(t)-2t\.h'(t)^2\,,\qquad b(t)= \left(t^2+3\right)h'(t)+2t\.(t^2+1)\.h''(t)\,,\\
		c(t)&=4t\left(h'(t)+t\.h''(t)\right)
	\end{align*}
	we obtain the equivalence\par
	\fbox{\parbox{0.99\textwidth}{$\begin{array}{lcr}\textbf{iv)}&\dd\sqrt{g_{xx}\,g_{yy}}+g_{xy}+\frac{g_x-g_y}{x-y}\geq 0\quad\text{for all }x,y,\in(0,\infty)\,,x\neq y\qquad\iff&\\
		&\hspace{-2.45em}\frac{2t}{t-1}\.h'(t)-t^2\.h''(t)+z^2\.f''(z)\geq 0\quad\text{or}\quad a(t)+\left[b(t)-c(t)\right]z^2\.f''(z)\geq 0\quad\text{for all }\;t,z\in(0,\infty)\,,\;t\neq 1\,.\hspace*{-.637em}&\textbf{C)}\end{array}$}}\par
	\textbf{Knowles and Sternberg v)}\quad $\dd\sqrt{g_{xx}\,g_{yy}}-g_{xy}+\frac{g_x+g_y}{x+y}\geq 0$\quad for all $x,y\in (0,\infty)$: It holds
	\begin{align}
		\sqrt{g_{xx}\,g_{yy}}\geq g_{xy}-\frac{g_x+g_y}{x+y}\quad\iff\quad -g_{xy}+\frac{g_x+g_y}{x+y}\geq 0\quad\text{or}\quad g_{xx}\,g_{yy}\geq \left(-g_{xy}+\frac{g_x+g_y}{x+y}\right)^2.\label{eq:condition5cases}
	\end{align}
	For the first case in equation \eqref{eq:condition5cases} we compute
	\begin{align}
		-g_{xy}+\frac{g_x+g_y}{x+y}=\frac{1}{z}\left[\frac{2t}{t+1}\.h'(t)+t^2\.h''(t)\right]-z\.f''(z)\geq 0\quad\overset{\cdot z}{\iff}\quad\frac{2t}{t+1}\.h'(t)+t^2\.h''(t)-z^2\.f''(z)\geq 0\,.
	\end{align}
	For the second case in equation \eqref{eq:condition5cases} we start with
	\begin{align}
		\left(-g_{xy}+\frac{g_x+g_y}{x+y}\right)^2&=\left(\frac{1}{z}\left(\frac{2t}{t+1}\.h'(t)+t^2\.h''(t)\right)-z\.f''(z)\right)^2\notag\\
		&=\frac{1}{z^2}\left(\frac{4t}{(t+1)^2}\.h'(t)^2+\frac{4t^3}{t+1}\.h'(t)h''(t)+t^4\.h''(t)^2\right)\\
		&\phantom{=}\;-\left(\frac{4t}{t+1}\.h'(t)+2t^2\.h''(t)\right)f''(z)+z^2f''(z)^2\notag
	\end{align}
	and compute
	\begin{align}
		&&&g_{xx}\,g_{yy}\geq \left(-g_{xy}+\frac{g_x+g_y}{x+y}\right)^2\notag\\
		&\iff&&\frac{1}{z^2}\left[2t^3\.h'(t)h''(t)+t^4\.h''(t)^2\right]+\left[2t\.h'(t)+2t^2\.h''(t)\right]f''(z)+z^2f''(z)^2\notag\\
		&&&\geq \frac{1}{z^2}\left[\frac{4t^2}{(t+1)^2}\.h'(t)^2+\frac{4t^3}{t+1}\.h'(t)h''(t)+t^4\.h''(t)^2\right]-\left[\frac{4t}{t+1}\.h'(t)+2t^2\.h''(t)\right]f''(z)+z^2f''(z)^2\notag\\
		&\overset{\mathclap{\cdot z^2}}{\iff}&&\left(2t^3-\frac{4t^3}{t+1}\right)h'(t)h''(t)-\frac{4t^2}{(t+1)^2}\.h'(t)^2+\left[\left(2t+\frac{4t}{t+1}\right)h'(t)+4t^2\.h''(t)\right]z^2\.f''(z)\geq 0\notag\\
		&\overset{\mathclap{\cdot\frac{(t+1)^2}{2t}}}{\iff}&& t^2(t^2-1)\.h'(t)h''(t)-2t\.h'(t)^2+\left[\left(t^2+4t+3\right)h'(t)+2t\.(t+1)^2\.h''(t)\right]z^2\.f''(z)\geq 0
	\end{align}
	Thus with the previous definitions
	\begin{align*}
		a(t)&= t^2(t^2-1)\.h'(t)h''(t)-2t\.h'(t)^2\,,\qquad b(t)= \left(t^2+3\right)h'(t)+2t\.(t^2+1)\.h''(t)\,,\\
		c(t)&=4t\left(h'(t)+t\.h''(t)\right)\,,
	\end{align*}
	we find	\par
	\fbox{\parbox{0.99\textwidth}{$\begin{array}{lcr}\textbf{v)}&\dd\sqrt{g_{xx}\,g_{yy}}-g_{xy}+\frac{g_x+g_y}{x+y}\geq 0\quad\text{for all }x,y,\in(0,\infty)\qquad\iff&\\
		&\frac{2t}{t+1}\.h'(t)+t^2\.h''(t)-z^2\.f''(z)\geq 0\quad\text{or}\quad a(t)+\left[b(t)+c(t)\right]z^2\.f''(z)\geq 0\quad\text{for all } t,z\in(0,\infty)\,,&\textbf{D)}\end{array}$}}\par
	thereby establishing the last equivalence and completing the proof of Lemma \ref{lemma:voliso}.
\end{proof}
%
%
%
%
%
\subsection{Reduction to a family of coupled one-dimensional differential inequalities}
In the following theorem we show that, surprisingly, it is possible to reduce the two-dimensional problem of rank-one convexity for planar objective-isotropic energies with additive volumetric-isochoric split to only coupled one-dimensional conditions.
\begin{theorem}\label{theorem:main}
	Let $W\col\GLp(2)\to\R$ be an objective-isotropic function of class $C^2$ with additive volumetric-isochoric split with the representation in terms of the singular values of the deformation gradient $F$ via $W(F)=g(\lambda_1,\lambda_2)=h\big(\frac{\lambda_1}{\lambda_2}\big)+f(\lambda_1\lambda_2)$, where $f,h\in C^2((0,\infty),\R)$, and let $h_0=\inf_{t\in(0,\infty)}t^2\.h''(t)$ and $f_0=\inf_{z\in(0,\infty)}z^2\.f''(z)$. Then $W$ is rank-one convex if and only if
	\begin{itemize}
		\item[1)] \qquad$\dd h_0+f_0\geq 0$\,,
		\item[2)] \qquad$\dd h'(t)\geq 0\qquad\text{for all}\quad t\geq 1$\,,
		\item[3)]\qquad$\dd \frac{2t}{t-1}\.h'(t)-t^2\.h''(t)+f_0\geq 0$\quad or\quad $\dd a(t)+\left[b(t)-c(t)\right]f_0\geq 0$\qquad$\text{for all}\quad t\in(0,\infty)\,,\;t\neq 1$\,,\\ 
		\item[4)]\qquad$\dd \frac{2t}{t+1}\.h'(t)+t^2\.h''(t)-f_0\geq 0\quad\text{or}\quad a(t)+\left[b(t)+c(t)\right]f_0\geq 0\qquad\text{for all}\quad t\in(0,\infty)$\,,
	\end{itemize}
	where
	\begin{align*}
		a(t)&=t^2(t^2-1)\.h'(t)h''(t)-2t\.h'(t)^2,\qquad b(t)=\left(t^2+3\right)h'(t)+2t\.(t^2+1)\.h''(t)\,,\\
		c(t)&=4t\left(h'(t)+t\.h''(t)\right).
	\end{align*}
\end{theorem}
\begin{proof}
	First, let the conditions 1)--4) be satisfied for some energy $W$. Choose a function $\fhat\col(0,\infty)\to\R$ such that $z^2\.\fhat''(z)\equiv f_0$ and let $\What(F)\colonequals h\big(\frac{\lambda_1}{\lambda_2}\big)+\fhat(\lambda_1\lambda_2)$. Then conditions 1)--4) for the original energy $W$ are equivalent to conditions A)--D) in Lemma \ref{lemma:voliso} for the energy $\What$. Therefore, $\What$ is rank-one convex. For $\ftilde(z)\colonequals f(z)-\fhat(z)$, we also find
	\begin{align}
		\inf_{z\in(0,\infty)}z^2\.\ftilde''(z)&=\inf_{z\in(0,\infty)}\left(z^2\.f''(z)-z^2\.\fhat''(z)\right)=f_0-f_0=0\notag\\
		\implies\qquad\qquad \ftilde''(z)&\geq 0\qquad\text{for all }\; z\in(0,\infty)\,,
	\end{align}
	thus $\ftilde$ is convex. Finally, $W$ can be decomposed into the sum
	\begin{align*}
		W(F)=h\bigg(\frac{\lambda_1}{\lambda_2}\bigg)+f(\det F)=h\bigg(\frac{\lambda_1}{\lambda_2}\bigg)+\fhat(\lambda_1\lambda_2)+f(\det F)-\fhat(\det F)=\What(F)+\ftilde(\det F)
	\end{align*}
	of the rank-one convex energy $\What$ and the $F\mapsto\ftilde(\det F)$, which is rank-one convex due to the convexity of $\ftilde$. Therefore, the energy $W$ is rank-one convex.
	
	Now, let $W$ be rank-one convex. We observe that for given $h$ and arbitrary fixed $t>0$, each of the conditions A)--D) of Lemma \ref{lemma:voliso} can be written as an inequality $P_{t,h}(z^2f''(z))\geq0$ for some continuous function $P_{t,h}$. Furthermore, using the same functions, each of the conditions 1)--4) can be expressed in the form $P_{t,h}(f_0)\geq0$. If $W$ is rank-one convex, i.e.\ if the conditions A)--D) are satisfied for all $z\in(0,\infty)$, then
	\[
		P_{t,h}(f_0) = P_{t,h}\left(\inf_{z\in(0,\infty)}z^2f''(z)\right) \geq \inf_{z\in(0,\infty)} P_{t,h}\left(z^2f''(z)\right) \geq 0
	\]
	due to continuity, therefore conditions 1)--4) hold as well.
%
\end{proof}
\begin{corollary}\label{lemma:bcgeqzero}
	Rank-one convexity of the energy $W(F)=h\bigl(\frac{\lambda_1}{\lambda_1}\bigr)+f(\lambda_1\.\lambda_2)$ implies
	\[
		b(t)+c(t)=\left(t^2+4\.t+3\right)h'(t)+2t\.(t+1)^2\.h''(t)\geq 0
	\]
	for all $t\in(0,\infty)$. 
\end{corollary}
\begin{proof}
	For arbitrary $\eta>0$, consider the rank-one convex class of energy functions
	\begin{align}
		W_\eta(F)\colonequals W(F)+\fhat_\eta(z)\,,
	\end{align}
	such that the convex volumetric part $\fhat_\eta\col(0,\infty)\to\R$ satisfies $z^2\.\fhat_\eta''(z)\equiv\eta$. Then
	\begin{align}
		\inf_{z\in(0,\infty)}z^2\left(f''(z)+\fhat_\eta''(z)\right)=f_0+\eta\,.
	\end{align}
	If $b(t_0)+c(t_0)<0$ for some arbitrary $t_0\in(0,\infty)$, then condition 4) of Theorem \ref{theorem:main} for the energy $W_\eta(F)$, which reads
	\begin{align}
		\frac{2\.t_0}{t_0+1}\.h'(t_0)+t_0^2\.h''(t_0)-f_0\underbrace{-\eta}_{<0}\geq 0\qquad\text{or}\qquad a(t_0)+\underbrace{\left[b(t_0)+c(t_0)\right]}_{<0}\big(f_0\underbrace{+\eta}_{>0}\big)\geq 0\,,
	\end{align}
	cannot be satisfied for $\eta\to\infty$, in contradiction to the rank-one convexity of the energy function $W_\eta(F)$ for all $\eta>0$. Thus $b(t)+c(t)\geq 0$ for all $t\in(0,\infty)$.
\end{proof}
\subsection{Further necessary conditions for rank-one convexity}
We continue the discussion with a list of necessary conditions for rank-one convexity of an energy with volumetric-isochoric split in terms of the isochoric part $h$ and volumetric part $f$.
\begin{lemma}\label{lemma:necessary}
	Consider an arbitrary isotropic rank-one convex energy function $W\col\GLp(2)\to\R$ with additive volumetric-isochoric split $W(F)=h\bigl(\frac{\lambda_1}{\lambda_2}\bigr)+f(\lambda_1\lambda_2)$. Then
	\begin{itemize}
		\item[a)] $h$ is globally convex or $f$ is globally convex.
		\item[b)] the mapping $x\mapsto h(x)+f(x)$ is globally convex on $(0,\infty)$.
		\item[c)] $h'(t)\geq 0$ for all $t>1$ and $h'(t)\leq 0$ for all $t<1$, i.e.\ $h$ is monotone increasing on $[1,\infty)$ and monotone decreasing on $(0,1]$.
		\item[d)] $t\.h''(t)+h'(t)\geq 0$ for all $t>0$. The inequality is strict if the energy is strictly rank-one convex.
		\item[e)] $\left(t+3\right)h'(t)+2t\.(t+1)\.h''(t)\geq 0$ for all $t>0$.
	\end{itemize}
\end{lemma}
\begin{proof}
	\textbf{Condition a)} The separate convexity condition 1) in Theorem \ref{theorem:main} states
	\begin{align}
		h_0+f_0\geq 0\qquad\implies\qquad &h_0\geq 0\quad \text{or}\quad f_0\geq 0\notag\\
		\iff\qquad &t^2\.h''(t)\geq 0\quad\text{for all }t\in(0,\infty)\quad \text{or}\quad z^2\.f''(z)\geq 0\quad\text{for all }z\in(0,\infty)\notag\\
		\iff\qquad &h''(t)\geq 0\quad\text{for all }t\in(0,\infty)\quad \text{or}\quad f''(z)\geq 0\quad\text{for all }z\in(0,\infty)\,.
	\end{align}
	\textbf{Condition b)} For any $a>0$, rank-one convexity of $W$ implies the convexity of $W$ along the straight line $t\mapsto\diag(a,1)+t\diag(1,0)$ in rank-one direction $\diag(1,0)$, i.e.\ the convexity of the mapping
	\begin{align}
		t\mapsto W(\diag(a+t,1))=g(a+t,1)=h\bigg(\frac{a+t}{1}\bigg)+f\big((a+t)\cdot 1\big)=h(a+t)+f(a+t)
	\end{align}
	and thus the convexity of the mapping $x\mapsto h(x)+f(x)$.\par
	\textbf{Condition c)} The monotonicity condition is already known as condition 2) in Theorem \ref{theorem:main} (as well as Lemma \ref{lemma:voliso}).\par
	\textbf{Condition d)} The inequality can be obtained by considering the restriction $\WSL=W|_{\SL(2)}$ of $W$ to the special linear group $\{F\in\GLp(2) \setvert \det F = 1\}$. Then rank-one convexity of $W$ implies rank-one convexity \cite{agn_martin2019quasiconvex} of $\WSL$ with respect to $\SL(2)$. Recall that any function $\WSL\col\SL(2)\to\R$ can be expressed in the form \eqref{eq:incompressibleEnergyRepresentation}, i.e.
	\[
		\WSL(F) = \phi\left(\lmax-\frac{1}{\lmax}\right)=\phi\left(\frac{\lmax}{\sqrt{\lmax\lmin}}-\frac{\sqrt{\lmax\lmin}}{\lmax}\right)=\phi\left(\sqrt{\frac{\lmax}{\lmin}}-\sqrt{\frac{\lmin}{\lmax}}\right)
	\]
	with $\phi\col[0,\infty)\to\R$, and $\WSL$ is rank-one convex on $\SL(2)$ if and only if $\phi$ is monotone and convex \cite{agn_martin2019quasiconvex}. For $\WSL=W|_{\SL(2)}$, we find
	\[
		\phi\left(\sqrt{\frac{\lmax}{\lmin}}-\sqrt{\frac{\lmin}{\lmax}}\right) = \WSL(F) = W(F) = h\bigg(\frac{\lmax}{\lmin}\bigg) + f(1)
	\]
	for all $F\in\SL(2)$ with singular values $\lmax>\lmin$ and thus
	\[
		h(t) 
		= \begin{cases}
			\phi\left(\sqrt{t}-\frac{1}{\sqrt{t}}\right) - f(1) &: t\geq1\,,\\
			\phi\left(\frac{1}{\sqrt{t}}-\sqrt{t}\right) - f(1) &: t<1\,.
		\end{cases}
	\]
	Then for any $t\geq1$,
	\begin{align}
		h'(t)&=\ddt \phi\left(\sqrt{t}-\frac{1}{\sqrt{t}}\right)=\phi'\left(\sqrt{t}-\frac{1}{\sqrt{t}}\right)\.\left(\frac{1}{2\.\sqrt{t}}+\frac{1}{2\.\sqrt{t^3}}\right),\\
		h''(t)&=\ddt\left[\frac{1}{2}\.\phi'\left(\sqrt{t}-\frac{1}{\sqrt{t}}\right)\.\left(\frac{1}{\sqrt{t}}+\frac{1}{\sqrt{t^3}}\right)\right]\notag\\
		&=\frac{1}{4}\.\phi''\left(\sqrt{t}-\frac{1}{\sqrt{t}}\right)\left(\frac{1}{\sqrt{t}}+\frac{1}{\sqrt{t^3}}\right)^2+\frac{1}{2}\phi'\left(\sqrt{t}-\frac{1}{\sqrt{t}}\right)\left(-\frac{1}{2\.\sqrt{t^3}}-\frac{3}{2\.\sqrt{t^5}}\right)
	\end{align}
	and thus
	\begin{align}
		t\.h''(t)+h'(t)&=\frac{1}{4}\.\phi''\left(\sqrt{t}-\frac{1}{\sqrt{t}}\right)t\left(\frac{1}{\sqrt{t}}+\frac{1}{\sqrt{t^3}}\right)^2+\frac{1}{2}\phi'\left(\sqrt{t}-\frac{1}{\sqrt{t}}\right)\left(-\frac{1}{2\.\sqrt{t}}-\frac{3}{2\.\sqrt{t^3}}+\frac{1}{\sqrt{t}}+\frac{1}{\sqrt{t^3}}\right)\notag\\
		&=
			\frac{1}{4}\.\underbrace{\phi''\left(\sqrt{t}-\frac{1}{\sqrt{t}}\right)}_{\geq 0 \text{ ($\phi$ convex)}}
			\cdot\,\underbrace{t\left(\frac{1}{\sqrt{t}}+\frac{1}{\sqrt{t^3}}\right)^2}_{>0 \text{ for } t>0}
			+ \frac{1}{4} \underbrace{\phi'\left(\sqrt{t}-\frac{1}{\sqrt{t}}\right)}_{\geq 0 \text{ ($\phi$ monotone)}}
			\cdot\,\underbrace{\left(\frac{1}{\sqrt{t}}-\frac{1}{\sqrt{t^3}}\right)}_{\geq 0 \text{ for } t\geq 1}\geq 0
	\end{align}
	due to the rank-one convexity of $\WSL$. For $t\leq 1$, the inequality can be obtained analogously. If the energy is strictly rank-one convex, it holds $\phi''(x)>0$ for all $x>0$ and thus $t\.h''(t)+h'(t)>0$ for all $t>0$.\par
	\textbf{Condition e)} The inequality follows directly from Lemma \ref{lemma:bcgeqzero} where it is shown that rank-one convexity implies
	\begin{align}
		b(t)+c(t)&=\left(t^2+3\right)h'(t)+2t\.(t^2+1)\.h''(t)+4t\left(h'(t)+t\.h''(t)\right)\notag\\
		&=\left(t^2+4t+3\right)h'(t)+2t\.(t+1)^2\.h''(t)=(t+1)\.\left[\left(t+3\right)h'(t)+2t\.(t+1)\.h''(t)\right]\geq 0\\
		\iff\qquad &\left(t+3\right)h'(t)+2t\.(t+1)\.h''(t)\geq 0\qquad\text{for all } t\in(0,\infty)\,.\hspace{10cm}\qedhere\hspace{-5cm}
	\end{align}
\end{proof}
%
%
%
%
%
\subsection{Application to generalized Hadamard energies}
In particular, our rank-one characterization applies to the family of planar isotropic functions
\begin{align}
	W\col\GLp(2)\to\R\,,\quad W(F)=\underbrace{\mu\,\K(F)}_{\mathclap{\text{rank-one convex}}}+f(\det F)\,,\quad\;\;\mu>0\,,\;\;\;\K(F)\colonequals\frac{1}{2}\.\frac{\norm{F}^2}{\det F}=\frac{1}{2}\left\|\frac{F}{\sqrt{\det F}}\right\|^2,\label{eq:muKenergy}
\end{align}
Here, $\K\geq 1$ is the so-called \emph{distortion function}, which plays an important role in the theory of quasiconformal mappings \cite{astala2008elliptic,agn_martin2019envelope}. Note that $\K\equiv 1$ if and only if $\varphi$ is conformal. It is well known that the mapping $F\mapsto\K(F)$ itself is rank-one convex and even polyconvex on $\GLp(2)$. Using our characterization for rank-one convexity of volumetric-isochoric split energies, we obtain the following lemma.
\begin{lemma}\label{lemma:Kenergy}
	The energy function $W\col\GLp(2)\to\R$ with $W(F)=\mu\,\K(F)+f(\det F)$ is rank-one convex if and only if $f$ is convex.
\end{lemma}
\begin{proof}
	We first observe that
	\begin{align*}
		\K(F)&=\frac{1}{2}\.\frac{\norm{F}^2}{\det F}=\frac{1}{2}\.\frac{\lambda_1^2+\lambda_2^2}{\lambda_1\lambda_2}=\frac{1}{2}\left(\frac{\lambda_1}{\lambda_2}+\frac{\lambda_2}{\lambda_1}\right)\\
		\implies\quad W(F) &= h\left( \frac{\lambda_1}{\lambda_2} \right) + f(\lambda_1\.\lambda_2) \qquad\text{with}\quad h(t)=\frac{1}{2}\left(t+\frac{1}{t}\right)
	\end{align*}
	and calculate the derivatives $h'(t)=\frac{1}{2}\left(1-\frac{1}{t^2}\right)$ and $h''(t)=\frac{1}{t^3}\,.$ Thus $h$ is globally convex and (strictly) increasing for $t>1$, hence the isochoric term $\mu\.\K$ is rank-one convex. If the volumetric term $f$ is convex as well, the energy function $W$ is rank-one convex as a sum of two rank-one convex functions.
	
	Now, in order to show the reverse implication, let $W$ be rank-one convex. Then $W$ is separately convex, which is expressed by condition 1) of Theorem \ref{theorem:main} as $h_0+f_0\geq 0$. We compute
	\begin{align*}
		h_0=\inf_{t\in(0,\infty)}t^2\.h''(t)=\inf_{t\in(0,\infty)}t^2\.\frac{1}{t^3}=\inf_{t\in(0,\infty)}\frac{1}{t}=0\,,
	\end{align*}
	and thus
	\begin{align*}
		f_0=\inf_{z\in(0,\infty)}z^2\.f''(z)\geq 0.
	\end{align*}
	Therefore, the function $f$ must be convex (which implies that the volumetric part $F\mapsto f(\det F)$ is rank-one convex).
\end{proof}
The result of Lemma \eqref{lemma:Kenergy} should be compared with the particular planar and quadratic case \cite[Theorem 5.58ii)]{Dacorogna08,ball1984w1,grabovsky2018rank} $W\col\GLp(2)\to\R$ of the Hadamard-material, for which
\begin{align}
	W(F)=\frac{\mu}{2}\norm{F}^2+f(\det F)\qquad\text{rank-one convex}\qquad\iff\qquad f\text{ is convex}\,,\;\mu>0\,.
\end{align}
In an intriguing recent contribution \cite{grabovsky2018explicit,grabovsky2016legendre}, the quasiconvex envelope for the Hadamard material has been computed analytically for a non-convex function $f$ and sufficiently large shear modulus $\mu$. It would be interesting to obtain a similar relaxation result for $W(F)=\frac{\mu}{2}\.\frac{\norm{F}^2}{\det F}+f(\det F)$.
%
%
%
%
%
\subsection{Idealized planar isotropic nonlinear energy function}
From a modeling perspective, it is a useful idealization to assume that the isochoric and volumetric response are governed by the same functional form. Consider an energy function $W\col\GLp(2)\to\R$ of the form
\begin{equation}
\label{eq:energyh}
	W(F) = \mu\.h\bigg(\frac{\lambda_1}{\lambda_2}\bigg) + \frac{\kappa}{2}\.h(\lambda_1\lambda_2)\,,
\end{equation}
where $h\in C^2((0,\infty);\R)$ with $h(t)=h\bigl(\frac{1}{t}\bigr)$ for all $t\in(0,\infty)$ and $\mu,\kappa>0$. This class of energies includes, for example, the Hencky energy \eqref{eq:hencky} and the exponentiated Hencky energy \eqref{eq:expHencky}. Note that any energy $W$ of this form is \emph{tension-compression symmetric}, i.e.\ satisfies $W(F)=W(F\inv)$ for all $F\in\GLp(2)$.

Theorem \ref{theorem:main} immediately shows that an energy of the form \eqref{eq:energyh} is rank-one convex if and only if $h$ is convex. In this case, the volumetric and isochoric parts are rank-one convex, individually. Thus for this particular type of split energy function, rank-one convexity is indeed equivalent to rank-one convexity of both the volumetric and the isochoric part. Moreover, according to the results of Section \ref{section:Cauchy}, the corresponding Cauchy stress-stretch response is locally invertible if $h$ is uniformly convex.
%
%
%
%
%
\subsection{Examples of non-trivial rank-one convex energies}\label{sec:examples}
In contrast to the class \eqref{eq:muKenergy} of functions, it is possible to find rank-one convex energies with either a non rank-one convex volumetric part $\Wvol$ or non rank-one convex isochoric part $\Wiso$. As an example, consider the energy
\begin{align}
	W\col\GLp(2)\to\R\,,\quad W(F)=e^{\frac{1}{10}\.(\log t)^2}+\frac{1}{60}\left(z-\frac{1}{z}\right)^2,\qquad t=\frac{\lambda_1}{\lambda_2}\,,\quad z=\det F\,.\label{eq:exampleEnergy}
\end{align}
The isochoric part $\Wiso=e^{\frac{1}{10}\left(\log \frac{\lambda_1}{\lambda_2}\right)^2}$ is not rank-one convex, cf. Figure \ref{fig:notRankOneconvex}. However, with formula \eqref{eq:linformula} the infinitesimal shear modulus\footnote{Here, $\Psi$ is given by the representation $W(F)=\psi\big(\K(F)\big)+f(\det F)$ of $W$ with $\psi\col[1,\infty)\to\R$ and $\K(F)=\frac{1}{2}\.\frac{\norm{F^2}}{\det F}$. In our example, $\psi(\K)=e^{\frac{1}{10}\left[\arcosh\K(F)\right]^2}$ and $\psi'(\K)=\frac{1}{5}\.e^{0.1\left[\arcosh\K(F)\right]^2}\frac{\arcosh\K(F)}{\sqrt{\K^2-1}}$.} $\mu=\psi'(1)=\frac{1}{5}>0$ and the infinitesimal bulk modulus $\kappa=f''(1)=\frac{2}{15}>0$ imply rank-one convexity in the infinitesimal case. Similar to the non rank-one convex energy in Figure \ref{fig:notRankOneconvex}, numerical results for the Knowles-Sternberg criterion directly suggest that the energy \eqref{eq:exampleEnergy} is everywhere rank-one convex on $\GLp(2)$. In order to prove rank-one convexity with Theorem \ref{theorem:main}, we compute
\begin{align}
	f(z)&=\frac{1}{60}\left(z-\frac{1}{z}\right)^2,\qquad f'(z)=\frac{1}{30}\left(z-\frac{1}{z}\right)\left(1+\frac{1}{z^2}\right)=\frac{1}{30}\left(z-\frac{1}{z^3}\right),\notag\\
	z^2\.f''(z)&=\frac{z^2}{30}\left(1+\frac{3}{z^4}\right)=\frac{1}{30}\left(z^2+\frac{3}{z^2}\right)
\end{align}
and determine the constant $f_0$:
\begin{align}
	\ddz\.z^2\.f''(z)&=\frac{1}{15}\left(z-\frac{3}{z^3}\right)\overset{!}{=}0\qquad\iff\qquad z^4-2=0\qquad\iff\qquad z=\sqrt[4]{3}\notag\\
	\implies\qquad f_0&\colonequals\inf_{z\in(0,\infty)}z^2\.f''(z)=\frac{1}{30}\left(\sqrt{3}+\frac{3}{\sqrt{3}}\right)=\frac{\sqrt{3}}{15}\approx0.11547\,.
\end{align}
Furthermore, numerical calculations yield $h_0\colonequals\inf_{t\in(0,\infty)}t^2\.h''(t)\approx-0.101677$, thus $h_0+f_0>0$. In addition, we compute
\begin{align*}
	h(t)=e^{\frac{1}{10}\.(\log t)^2},\qquad h'(t)=\frac{\log t}{5t}\.e^{\frac{1}{10}\.(\log t)^2}\geq 0\qquad\text{for all }t\geq 1\,.
\end{align*}
The remaining conditions 3) and 4) of Theorem \ref{theorem:main} are too complex to be solved analytically, but are numerically visualized in Figure \ref{fig:RankOneConvex}.
\begin{figure}[h!]
	\vspace{-1.5cm}
	\begin{minipage}[t]{1\linewidth}
	   \centering
	   \includegraphics[width=\textwidth]{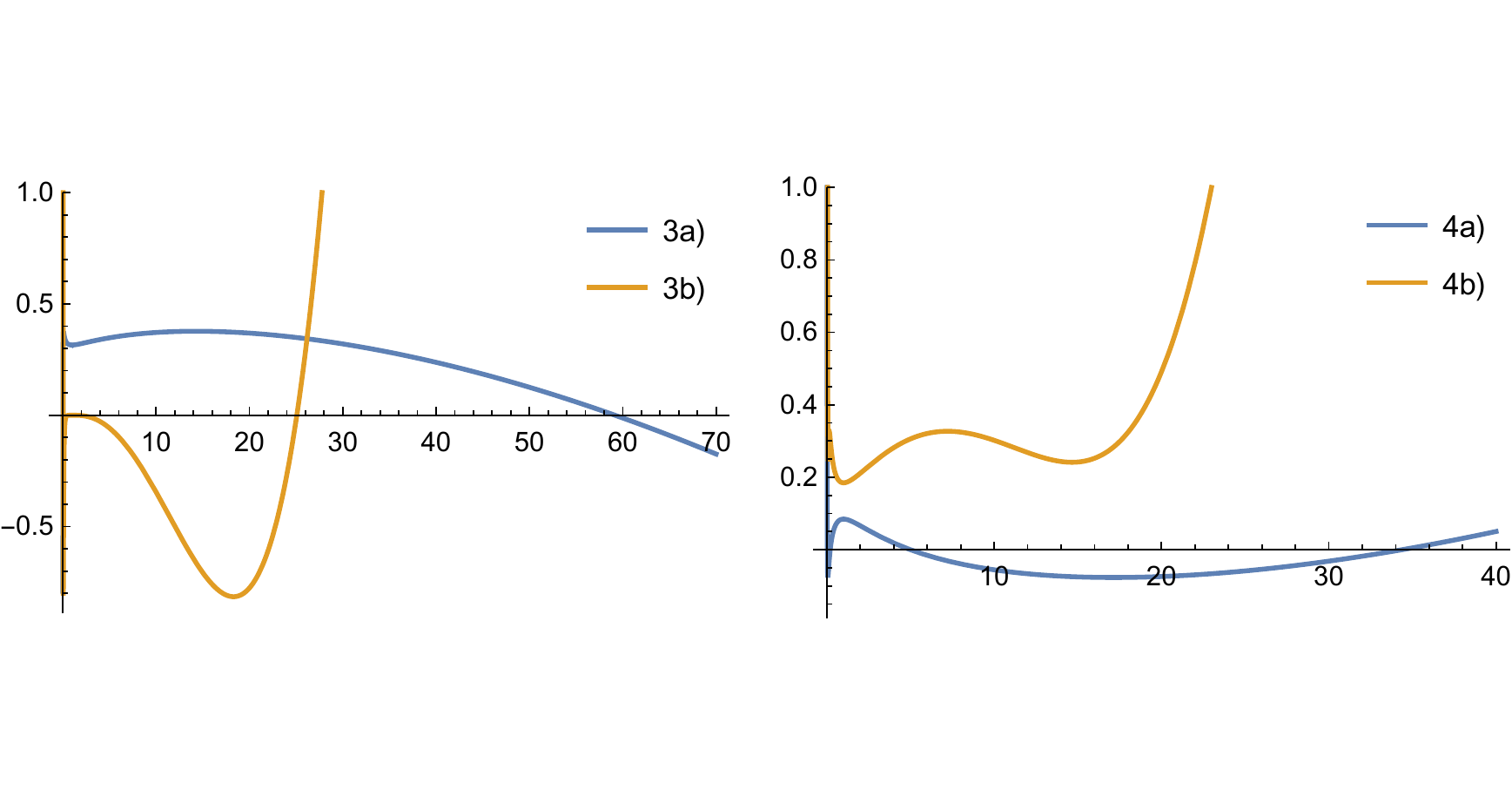}
	\end{minipage}
	\vspace{-2cm}
  \caption{\label{fig:RankOneConvex} Visualization of the rank-one convexity conditions 3) and 4) of Theorem \ref{theorem:main}. Note that for each condition and each $t\in(0,\infty)$, only one of the two terms must be positive, which is the case here.}
\end{figure}

As a second example, consider the energy
\begin{align}
	W\col\GLp(2)\to\R\,,\quad W(F)=\frac{6}{5}\left(t-\frac{1}{t}\right)^2+\left(z-\frac{1}{z}\right)^4-\left(z-\frac{1}{z}\right)^2,\qquad t=\frac{\lambda_1}{\lambda_2}\,,\quad z=\det F\label{eq:exampleEnergy2}
\end{align}
with a non rank-one convex volumetric part $\Wvol$ in the form of a double-well potential. Again, with formula \eqref{eq:linformula} the infinitesimal shear modulus\footnote{Here, $\psi(\K)=\frac{6}{5}\left(\K+\sqrt{\K^2-1}-\frac{1}{\K+\sqrt{\K^2-1}}\right)^2=\frac{24}{5}\left(\K^2-1\right)$ and $\psi'(\K)=\frac{48}{5}\.\K$
.} $\mu=\psi'(1)=\frac{48}{5}>0$ and the infinitesimal bulk modulus $\kappa=f''(1)=-8<0$ imply rank-one convexity in the infinitesimal case. Similar to the non rank-one convex energy in Figure \ref{fig:notRankOneconvex}, employing numerical calculations for the Knowles-Sternberg criterion directly suggests that the energy \eqref{eq:exampleEnergy2} is everywhere rank-one convex on $\GLp(2)$. Again, the rank-one convexity of $W$ can be shown using Theorem \ref{theorem:main}. First, we compute
\begin{align}
	f(z)&=\left(z-\frac{1}{z}\right)^4-\left(z-\frac{1}{z}\right)^2=\frac{1}{z^4}-\frac{5}{z^2}+8-5\.z^2+z^4\,,\qquad f'(z)=-\frac{4}{z^5}+\frac{10}{z^3}-10\.z+4\.z^3,\notag\\
	z^2\.f''(z)&=\frac{20}{z^4}-\frac{30}{z^2}-10\.z^2+12\.z^4\,.
\end{align}
For $z=1$, we obtain the infimum $f_0\colonequals\inf_{z\in(0,\infty)}z^2\.f''(z)=-8$, whereas $h_0\colonequals\inf_{t\in(0,\infty)}t^2\.h''(t)=\frac{24\sqrt{3}}{5}\approx 8.31384$ can be calculated similarly to the value of $f_0$ in Example \eqref{eq:exampleEnergy}. Thus $h_0+f_0>0$. In addition, we compute
\begin{align*}
	h(t)=\frac{6}{5}\left(t-\frac{1}{t}\right)^2,\qquad h'(t)=\frac{12}{5}\left(t-\frac{1}{t^3}\right)\geq 0\qquad\text{for all }t\geq 1\,.
\end{align*}
Again, the remaining conditions 3) and 4) of Theorem \ref{theorem:main} are too complex to allow for analytical solutions, but are easily verified numerically similar to the case shown in Figure \ref{fig:RankOneConvex}.
%
%
%
%
%
\newcommand{\ha}{h'\bigg(\frac{\lambda_1}{\lambda_2}\bigg)}
\newcommand{\hb}{h''\bigg(\frac{\lambda_1}{\lambda_2}\bigg)}
\newcommand{\fa}{f'(\lambda_1\lambda_2)}
\newcommand{\fb}{f''(\lambda_1\lambda_2)}
\section{Invertibility of the Cauchy stress tensor in planar elasticity}\label{section:Cauchy}
The volumetric-isochoric split allows for a simplified investigation of invertibility for the Cauchy stress-stretch law \cite{agn_schweickert2019nonhomogeneous} as well. First, recall that in planar linear isotropic elasticity, the energy potential and the Cauchy stress response are given by
\begin{align}
	\Wlin(\nabla u)&=\Wliniso(\nabla u)+\Wlinvol(\nabla u)=\mu\.\norm{\dev_2\sym\nabla u}^2+\frac{\kappa}{2}\.(\tr\nabla u)^2,\\
	\sigmalin(\nabla u)&=\sigmaliniso(\nabla u)+\sigmalinvol(\nabla u)=D\Wlin(\nabla u)=2\mu\.\dev_2\sym\nabla u+\kappa\.\tr(\nabla u)\.\id\,,\notag
\end{align}
respectively. Then \cite{agn_neff2016injectivity,agn_mihai2016hyperelastic,agn_mihai2018hyperelastic}
\begin{align}
	\Wlin\quad\text{strictly convex}\qquad\iff\qquad\mu>0\,,\;\kappa>0\qquad\iff\qquad\sigmalin\quad\text{invertible.}
\end{align}
\begin{lemma}
	Assume that $\Wlin$ is strictly rank-one convex and $\kappa>0$. Then $\mu>0$ (cf.\ Appendix \ref{appendix:lin}), $\kappa>0$, and the Cauchy stress-strain law is uniquely invertible.
\end{lemma}
Now, consider the planar nonlinear isotropic case with an additive volumetric-isochoric split.
\begin{lemma}\label{lemma:cauchyInvertibility}
	Assume that $W(F)=\Wiso\big(\frac{F}{\sqrt{\det F}}\big)+\Wvol(\det F)$ is strictly rank-one convex and $f$ is uniformly convex with $\Wvol(F)=f(\det F)$. Then the Cauchy stress-stretch law is locally invertible.
\end{lemma}
\begin{proof}
	Again, we use the notation $W(F)=g(\lambda_1,\lambda_2)=h\big(\frac{\lambda_1}{\lambda_2}\big)+f(\lambda_1\lambda_2)$
	and compute principle Cauchy stresses
	\begin{align}
		\sigmahat_1&=\frac{1}{\lambda_2}\.\frac{\partial W}{\partial\lambda_1}=\frac{1}{\lambda_2}\left[\frac{1}{\lambda_2}\.\ha+\lambda_2\.\fa\right]=\underbrace{\frac{1}{\lambda_2^2}\.\ha}_{\sigma_{\rm iso}}+\underbrace{\fa}_{\sigma_{\rm vol}}\,,\\
		\sigmahat_2&=\frac{1}{\lambda_1}\.\frac{\partial W}{\partial\lambda_2}=\frac{1}{\lambda_1}\left[-\frac{\lambda_1}{\lambda_2^2}\.\ha+\lambda_1\.\fa\right]=\underbrace{-\frac{1}{\lambda_2^2}\.\ha}_{-\sigma_{\rm iso}}+\underbrace{\fa}_{\sigma_{\rm vol}}\,.
	\end{align}
	Thus $\sigmahat$ has the form
	\begin{align}
		\sigmahat(\lambda_1,\lambda_2)=\matr{\phantom{-}\sigma_{\rm iso}(\lambda_1,\lambda_2)+\sigma_{\rm vol}(\lambda_1,\lambda_2)\\-\sigma_{\rm iso}(\lambda_1,\lambda_2)+\sigma_{\rm vol}(\lambda_1,\lambda_2)}\equalscolon\matr{\phantom{-}a+b\\-a+b}
	\end{align}
	with $a,b\in\R$. We compute
	\begin{align}
		\det(D\sigmahat)&=\left|\begin{matrix}
			a_{x}+b_{x} & a_{y}+b_{y} \\ -a_{x}+b_{x} & -a_{y}+b_{y}
		\end{matrix}\right|=(a_{x}+b_{x})\.(-a_{y}+b_{y})-(-a_{x}+b_{x})\.(a_{y}+b_{y})\\
		&=-a_{x}a_{y}+a_{x}b_{y}-a_{y}b_{x}+b_{x}b_{y}-\left[-a_{x}a_{y}-a_{x}b_{y}+a_{y}b_{x}+b_{x}b_{y}\right]= 2\left[a_{x}b_{y}-a_{y}b_{x}\right],\notag
	\end{align}
	where
	\begin{align*}
		a_{x}&=\frac{\partial}{\partial\lambda_1}\sigma_{\rm iso}=\frac{d}{d\lambda_1}\left[\frac{1}{\lambda_2^2}\.\ha\right]=\frac{1}{\lambda_2^3}\.\hb,\\
		a_{y}&=\frac{\partial}{\partial\lambda_2}\sigma_{\rm iso}=\frac{d}{d\lambda_2}\left[\frac{1}{\lambda_2^2}\.\ha\right]=-\frac{2}{\lambda_2^3}\.\ha+\frac{1}{\lambda_2^2}\.\hb\.\frac{-\lambda_1}{\lambda_2^2}=-\frac{1}{\lambda_2^3}\left[2\.\ha+\frac{\lambda_1}{\lambda_2}\.\hb\right],\\
		b_{x}&=\frac{\partial}{\partial\lambda_1}\sigma_{\rm vol}=\frac{d}{d\lambda_1}\.\fa=\lambda_2\.\fb\,,\qquad b_{y}=\frac{\partial}{\partial\lambda_2}\sigma_{\rm vol}=\frac{d}{d\lambda_2}\.\fa=\lambda_1\.\fb\,.
	\end{align*}
	Thus
	\begin{align}
		\det(D\sigmahat)&=2\left[\frac{1}{\lambda_2^3}\.\hb\cdot\lambda_1\.\fb+\frac{1}{\lambda_2^3}\left(2\.\ha+\frac{\lambda_1}{\lambda_2}\.\hb\right)\cdot\lambda_2\.\fb \right]\notag\\
		&=\frac{4\.\fb}{\lambda_2^2}\left[\frac{\lambda_1}{\lambda_2}\.\hb+\ha \right].
	\end{align}
	The stress-strain response is locally invertible if
	\begin{align}
		\det(D\sigmahat)=\frac{4\.\fb}{\lambda_2^2}\left[\frac{\lambda_1}{\lambda_2}\.\hb+\ha \right]&>0\qquad\text{for all }\;\lambda_1,\lambda_2>0\notag\\
		\iff\qquad f''(z)\left[t\.h''(t)+h'(t)\right]&>0\qquad\text{for all }\;t,z\in(0,\infty) \label{eq:stressStrainInvertibilityFinalInequality}
	\end{align}
	with $t=\frac{\lambda_1}{\lambda_2}$ and $z=\lambda_1\lambda_2$, with inequality \eqref{eq:stressStrainInvertibilityFinalInequality} being satisfied due to the implications
	\begin{align*}
		f\quad \text{is uniformly convex}\qquad&\implies \qquad f''(z)>0\qquad\text{for all }z\in(0,\infty)\\
		\text{and}\qquad W\quad \text{is strictly rank-one convex}\qquad &\overset{\mathclap{\text{Lemma }\ref{lemma:necessary}d)}}{\implies} \qquad  t\.h''(t)+h'(t)>0\qquad\text{for all }t\in(0,\infty)\,.\qedhere
	\end{align*}
\end{proof}
\begin{corollary}
	Let $W\col\GLp(2)\to\R$ with $W(F)=h\bigl(\frac{\lambda_1}{\lambda_2}\bigr)+f(\det(F))$ for all $F\in\GLp(2)$ with singular values $\lambda_1,\lambda_2$. If $f''(z)>0$ for all $z\in(0,\infty)$ and $t\.h''(t)+h'(t)>0$ for all $t\in(0,\infty)$, then the Cauchy stress-stretch law is locally invertible.
\end{corollary}
\begin{corollary}
	Let $W\col\GLp(2)\to\R$ with $W(F)=h\bigl(\frac{\lambda_1}{\lambda_2}\bigr)+f(\det F)$ for all $F\in\GLp(2)$ with singular values $\lambda_1,\lambda_2$ for uniformly convex functions $h,f\col(0,\infty)\to\R$. Then the Cauchy stress-stretch law is locally invertible.\footnote{The statement follows from the rank-one convexity of $W$ and Lemma \ref{lemma:cauchyInvertibility}. Alternatively, the inequality $t\.h''(t)+h'(t)>0$ can be shown directly by observing that $h'(1)=0$ and $h''(t)>0$ for all $t\geq 1$ implies
	\begin{align*}
		h'(t)\geq 0 \quad\text{for all }t\geq 1\qquad\implies\qquad t\.h''(t)+h'(t)>0\quad\text{for all }t\geq 1\,.
	\end{align*}
	For the case $t<1$, we use the symmetry $h(t)=h\big(\frac{1}{t}\big)$ and compute
	\begin{align}
		h'(t)&=\ddt\. h(t)=\ddt\.h\left(\frac{1}{t}\right)=-\frac{1}{t^2}\.h'\left(\frac{1}{t}\right),\notag\\
		h''(t)&=\frac{d^2}{dt^2}\.h(t)=\frac{d^2}{dt^2}\.h\left(\frac{1}{t}\right)=\ddt\left[-\frac{1}{t^2}\.h'\left(\frac{1}{t}\right)\right]=\frac{2}{t^3}\.h'\left(\frac{1}{t}\right)+\frac{1}{t^4}\.h''\left(\frac{1}{t}\right),\\
		t\.h''(t)+h'(t)&=\frac{2}{t^2}\.h'\left(\frac{1}{t}\right)+\frac{1}{t^3}\.h''\left(\frac{1}{t}\right)-\frac{1}{t^2}\.h'\left(\frac{1}{t}\right)=\underbrace{\frac{1}{t^2}\.h'\left(\frac{1}{t}\right)}_{\geq 0}+\underbrace{\frac{1}{t^3}\.h''\left(\frac{1}{t}\right)}_{>0}>0\quad\text{for all }t<1\,.\notag
	\end{align}}
\end{corollary}
\section{Conclusion}
Our investigation clearly demonstrates that for planar isotropic hyperelasticity, assuming a volumetric-isochoric split of the elastic energy endows the theory with a lot of additional mathematical structure which can be consequently exploited. Indeed, we have shown how the classical Knowles-Sternberg planar ellipticity criterion - which is represented by a  family of two-dimensional differential inequalities - can be reduced to a family of only one-dimensional coupled differential inequalities for split energies, which allows for a much more accessible rank-one convexity criterion.

By using these reduced inequalities, we could show that it is possible for a volumetric-isochorically split energy to be altogether rank-one convex even if the isochoric part is non-elliptic, in stark contrast to the linear case, where only the volumetric part of a rank-one convex energy might be non-elliptic.
We also applied our method to the general class of Hadamard-type materials and obtained a universal classification of their rank-one convexity.


In future contributions, based on encouraging preliminary results, we aim to extend our investigation from rank-one convexity to polyconvexity.
Unfortunately, our methods are at present strictly restricted to the planar isotropic case; it remains to be seen whether similar statements can be established for the three-dimensional case as well.
%
%
%
%
%
\subsubsection*{Acknowledgements}
The work of I.D.\ Ghiba was supported by a grant of the Romanian Ministry of Research and Innovation, CNCS--UEFISCDI, project number PN-III-P1-1.1-TE-2019-0397, within PNCDI III.
%
%
%
%
\newpage
\footnotesize
\section{References}
\printbibliography[heading=none]
\begin{appendix}
%
%
%
%
\section{Rank-one convexity in planar linear elasticity}\label{appendix:lin}
For the convenience of the reader, we recall here the well-known computation for rank-one convexity conditions in planar linear isotropic elasticity. Since \begin{align}
	\Wlin(\nabla u)&=\mu\.\norm{\dev_2\sym\nabla u}^2+\frac{\kappa}{2}\.(\tr\nabla u)^2
\end{align}
is a quadratic expression in terms of the displacement gradient $\grad u$ and hence $\frac{1}{2}\.D^2\Wlin(\nabla u).(H,H)=\Wlin(H)$, the energy $\Wlin$ is rank-one convex if and only if $\Wlin(\xi\otimes\eta)\geq 0$ for all $\xi,\eta\in\R^2$. We compute
\begin{align}
	\Wlin(\xi\otimes\eta)&=\mu\.\norm{\dev_2\sym\xi\otimes\eta}^2+\frac{\kappa}{2}\.(\tr\xi\otimes\eta)^2\notag\\
	&\overset{\mathclap{\eqref{eq:dev2}}}{=}\;\mu\left(\norm{\sym\xi\otimes\eta}^2-\frac{1}{2}\.(\tr\sym\xi\otimes\eta)^2\right)+\frac{\kappa}{2}\.(\tr\xi\otimes\eta)^2\notag\\
	&=\mu\left(\norm{\frac{1}{2}(\xi\otimes\eta+\eta\otimes\xi)}^2-\frac{1}{2}\.(\tr\xi\otimes\eta)^2\right)+\frac{\kappa}{2}\.(\tr\xi\otimes\eta)^2\notag\\
	&=\frac{\mu}{4}\left(\norm{\xi\otimes\eta}^2+2\.\iprod{\xi\otimes\eta,\eta\otimes\xi}+\norm{\eta\otimes\xi}^2\right)+\frac{\kappa-\mu}{2}\.(\tr\xi\otimes\eta)^2\\
	&=\frac{\mu}{2}\.\abs{\xi}^2\abs{\eta}^2+\frac{\mu}{2}\.\iprod{\xi,\eta}^2+\frac{\kappa-\mu}{2}\.\iprod{\xi,\eta}^2=\frac{\mu}{2}\.\abs{\xi}^2\abs{\eta}^2+\frac{\kappa}{2}\.\iprod{\xi,\eta}^2\,,\notag
\end{align}
where we used the equality
\begin{align*}
	\iprod{(\xi\otimes\eta),(\eta\otimes\xi)}=\iprod{(\eta\otimes\xi)\.(\xi\otimes\eta)^T,\id}=\iprod{(\xi\otimes\eta)\.(\xi\otimes\eta),\id}=\iprod{\iprod{\eta,\xi}\.\xi\otimes\eta,\id}=\iprod{\eta,\xi}\tr\xi\otimes\eta=\iprod{\eta,\xi}\.\iprod{\xi,\eta}=\iprod{\xi,\eta}^2.
\end{align*}
Therefore, due to the Cauchy-Schwarz inequality, the energy potential in planar isotropic linear elasticity is rank-one convex if and only if $\mu\geq 0$ and $\mu+\kappa\geq 0$, where $\mu$ and $\kappa$ represent the (planar) shear modulus and the bulk modulus, respectively. In terms of the first Lam\'e parameter $\lambda=\kappa-\mu$, rank-one convexity is also equivalent to the conditions $\mu\geq0$ and $2\mu+\lambda\geq 0$.

Since the rank-one convexity of the isochoric and the volumetric part of the energy potential in linear elasticity are determined\footnote{%
	Note that the infinitesimal isochoric term $\norm{\dev_2\sym\xi\otimes\eta}^2=\frac{1}{2}\.\abs{\xi}^2\abs{\eta}^2$ is itself strictly rank-one convex.%
}
by the signs of $\mu$ and $\kappa$, respectively, we find
\begin{align}
	\Wlin\quad\text{strictly rank-one convex}\qquad\iff\qquad\mu>0\,,\;\mu+\kappa>0\qquad\implies\qquad\Wliniso\quad\text{is strictly rank-one convex}
\end{align}
whereas the volumetric part $\Wlinvol$ is not necessarily elliptic by itself even if $\Wlin$ is rank-one convex.
%
%
%
%
\section{The planar Hadamard material}
We linearize the planar Hadamard material
\begin{align}
	W(F)=\frac{\mu}{2}\.\norm{F}^2+f(\det F)\,,\qquad f'(1)=-\mu\,,
\end{align}
which does not exhibit an additive volumetric-isochoric split, but is stress-free in the reference configuration under the assumption that $f'(1)=-\mu$. First, we compute
\begin{align}
	W(\id+H)&=\frac{\mu}{2}\.\norm{\id+H}^2+f(\det(\id+H))=\frac{\mu}{2}\left(\norm{\id}^2+2\.\tr H+\norm{H}^2\right)+f(1+\tr H+\det H)\notag\\
	&=\frac{\mu}{2}\left(2+2\.\tr H+\norm{H}^2\right)+f(1)+f'(1)\left(\tr H+\det H\right)+\frac{f''(1)}{2}\left(\tr H+\det H\right)^2+O(H^3)\\
	&=\underbrace{\mu+f(1)\vphantom{\frac{f''(1)}{2}}}_{\text{constant}\vphantom{f'\mu}}+\underbrace{\mu\.\tr H+f'(1)\.\tr H\vphantom{\frac{f''(1)}{2}}}_{f'(1)=-\mu\,,\;\text{stress-free}}+\underbrace{\frac{\mu}{2}\.\norm{H}^2+f'(1)\.\det H+\frac{f''(1)}{2}\.(\tr H)^2}_{\text{quadratic terms}\vphantom{f'\mu}}+O(H^3)\,.\notag
\end{align}
In order to continue with the linearization, we note that $(\tr H)^2-\tr(H^2)=2\.\det(H)$ for any $H\in\R^{2\times2}$ and
\begin{align*}
	\norm{\sym H}^2-\frac{1}{2}\.\norm{H}^2&=\iprod{\frac{1}{2}(H+H^T),\frac{1}{2}(H+H^T)}-\frac{1}{2}\.\norm{H}^2\\
	&=\frac{1}{4}\.\norm{H}^2+\frac{1}{2}\iprod{H,H^T}+\frac{1}{4}\.\norm{H^T}^2-\frac{1}{2}\.\norm{H}^2
	=\frac{1}{2}\iprod{H^2,\id}=\frac{1}{2}\.(\tr H^2)\,.
\end{align*}
Therefore
\begin{equation}
	\det H
	= \frac12\.(\tr H)^2 - \frac12\.\tr(H^2)
	= \frac12\.(\tr H)^2+\frac12\.\norm{H}^2-\norm{\sym H}^2\,,\label{eq:detPlanar}
\end{equation}
so we compute
\begin{align}
	W(\id+H)&=W(\id)+\frac{\mu}{2}\.\norm{H}^2-\mu\.\det H+\frac{f''(1)}{2}\.(\tr H)^2+O(H^3)\notag\\
	&=W(\id)+\frac{\mu}{2}\.\norm{H}^2-\frac{\mu}{2}\.(\tr H)^2-\frac{\mu}{2}\.\norm{H}^2+\mu\.\norm{\sym H}^2+\frac{f''(1)}{2}\.(\tr H)^2+O(H^3)\\
	&=W(\id)+\mu\.\norm{\sym H}^2+\left(\frac{f''(1)}{2}-\frac{\mu}{2}\right)(\tr H)^2+O(H^3)\,.\notag
\end{align}
Note that $\mu$ is indeed the infinitesimal shear modulus and that the infinitesimal Lam\'e-parameter is given by $\lambda=f''(1)-\mu$. Together with the well-known decomposition
\begin{align}
	\norm{\dev_2 X}^2=\norm{X-\frac{1}{2}\tr(X)\.\id}^2&=\norm{X}^2-\iprod{X,\tr(X)\.\id}+\frac{1}{4}\.\tr(X)^2\norm{\id}^2=\norm{X}^2-\left(\tr X\right)^2+\frac{1}{2}\left(\tr X\right)^2=\norm{X}^2-\frac{1}{2}\left(\tr X\right)^2\notag\\
	\iff\qquad\norm{X}^2&=\norm{\dev_2 X}^2+\frac{1}{2}\left(\tr X\right)^2\label{eq:dev2}
\end{align}
of the Frobenius norm of $X\in\R^{2\times2}$ into its deviatoric part and its trace part, we obtain
\begin{align}
	W(\id+H)&=W(\id)+\mu\.\norm{\sym H}^2+\left(\frac{f''(1)}{2}-\frac{\mu}{2}\right)(\tr H)^2=\mu\.\norm{\dev_2\sym H}^2+\frac{\mu}{2}\.(\tr H)^2+\left(\frac{f''(1)}{2}-\frac{\mu}{2}\right)(\tr H)^2+O(H^3)\notag\\
	&=W(\id)+\mu\.\norm{\dev_2\sym H}^2+\frac{f''(1)}{2}\.(\tr H)^2+O(H^3)\,,
\end{align}
thus the infinitesimal bulk modulus is $\kappa=f''(1)$. Overall,
\begin{align}
	\What_{\rm lin}(\eps)=\mu\.\norm{\eps}^2+\frac{f''(1)-\mu}{2}\.(\tr\eps)^2=\mu\.\norm{\dev_2\eps}^2+\frac{f''(1)}{2}\.(\tr\eps)^2\,,
\end{align}
where $\eps=\sym\nabla u$ denotes the infinitesimal strain.
%
%
%
%
%
\section{The planar Hadamard material with volumetric-isochoric split}
Next, we linearize the planar isotropic energy
\begin{align}
	\What(F)=\mu\,\K+f(\det F)=\frac{\mu}{2}\.\left\|\frac{F}{\sqrt{\det F}}\right\|^2+f(\det F)\,,\qquad \K=\frac{1}{2}\.\frac{\norm{F}^2}{\det F}\geq 1\,,
\end{align}
now with an additive volumetric-isochoric split, which is stress-free in the reference configuration if $f'(1)=0$. We start with the expansion
\begin{align}
	\K(\id+H)&=\frac{1}{2}\.\frac{\norm{\id+H}^2}{\det(\id+H)}=\frac{1}{2}\frac{\norm{\id}^2+2\.\tr H+\norm{H}^2}{\det\id+\tr H+\det H}=\frac{1}{2}\left(2+2\.\tr H+\norm{H}^2\right)\frac{1}{1+\tr H+\det H}\notag\\
	&=\left(1+\tr H+\frac{1}{2}\.\norm{H}^2\right)\left[1-\left(\tr H+\det H\right)+\left(\tr H+\det H\right)^2+O(H^3)\right]\notag\\
	&=\left(1+\tr H+\frac{1}{2}\.\norm{H}^2\right)\left(1-\tr H-\det H+(\tr H)^2+O(H^3)\right)\\
	&=1-\tr H-\det H+(\tr H)^2+(\tr H)\left(1-\tr H\right)+\frac{1}{2}\.\norm{H}^2+O(H^3)\notag\\
	&=1+(\tr H)^2+\frac{1}{2}\.\norm{H}^2-\det H+O(H^3)\notag\\
	&\overset{\mathclap{\eqref{eq:detPlanar}}}{=}\;1+(\tr H)^2+\frac{1}{2}\.\norm{H}^2-\left(\frac{1}{2}\.(\tr H)^2+\frac{1}{2}\.\norm{H}^2-\norm{\sym H}^2\right)+O(H^3)=1+\norm{\sym H}^2-\frac{1}{2}\.(\tr H)^2+O(H^3)\,,\notag
\end{align}
note that
\begin{align*}
	(1+h)\.(1-h+h^2+O(h^3))=1+h-h-h^2+h^2+h^3=1+O(h^3)\,.
\end{align*}
We also compute
\begin{align*}
	W(\id+H)&=\mu\,\K(\id+H)+f(1+\tr H+\det H)\\
	&=\mu+\mu\.\norm{\sym H}^2-\frac{\mu}{2}\.(\tr H)^2+f(1)+\underbrace{f'(1)\left(\tr H+\det H\right)}_{\mathclap{=0\,,\;\text{stress-free}}}+\frac{f''(1)}{2}\.(\tr H)^2+O(H^3)\\
	&=\mu+f(1)+\mu\.\norm{\sym H}^2+\frac{f''(1)-\mu}{2}\.(\tr H)^2+O(H^3)\overset{\eqref{eq:dev2}}{=}\mu+f(1)+\mu\.\norm{\dev_2\sym H}^2+\frac{f''(1)}{2}\.(\tr H)^2+O(H^3)
\end{align*}
and observe that $\mu$ is the infinitesimal shear modulus, $\lambda=f''(1)-\mu$ is the infinitesimal Lam\'e-parameter and $\kappa=f''(1)$ the infinitesimal bulk modulus. Overall,
\begin{align}
	\What_{\rm lin}(\eps)=\mu\.\norm{\eps}^2+\frac{f''(1)-\mu}{2}\.(\tr\eps)^2=\mu\.\norm{\dev_2\eps}^2+\frac{f''(1)}{2}\.(\tr\eps)^2
\end{align}
in terms of the infinitesimal strain $\eps=\sym\nabla u$.

Now, we consider the more general case
\begin{align}
	W(F)=\psi(\K(F))+f(\det F)
\end{align}
for some arbitrary smooth $\psi\col[1,\infty)\to\R$. Then the linearization of $W$ is given by
\begin{align}
	W(\id+H)&=\psi(\K(\id+H))+f(\det(\id+H))=\psi(1+\norm{\dev_2\sym H}^2+O(H^3))+f(1)+\frac{f''(1)}{2}\.(\tr H)^2+O(H^3)\notag\\
	&=\psi(1)+\psi'(1)\.\norm{\dev_2\sym H}^2+\psi''(1)\.\norm{\dev_2\sym H}^4+f(1)+\frac{f''(1)}{2}\.(\tr H)^2+O(H^3)\notag\\
	&=W(\id)+\psi'(1)\.\norm{\dev_2\sym H}^2+\frac{f''(1)}{2}\.(\tr H)^2+O(H^3)\,.\label{eq:linformula}
\end{align}
Therefore, the infinitesimal shear modulus and bulk modulus are given by $\mu=\psi'(1)$ and $\kappa=f''(1)$, respectively.
%
%
%
%
%
\section{Algebraic characterization of rank-one convexity}\label{appendix:dumitrel}
In order to illustrate the extent to which our results simplify the test for rank-one convexity of energy functions with an additive volumetric-isochoric split compared to a more direct approach, we provide here some of our initial computations which did not utilize the specific singular value representation in \eqref{eq:definitionOfh}.

Again, consider an energy function of the form
\begin{align}
	W\col\GLp(2)\to\R\,,\quad
	W(F)=\Wiso\bigg(\frac{F}{(\det F)^{1/2}}\bigg)+W_{\rm vol}(\det F)=\Wiso\bigg(\frac{F}{\sqrt{\det F}}\bigg)+f(\det F)\,.
\end{align}
We compute the bilinear form of the second derivative:
\begin{align}
	D_F^2W(F).(H,H)=D_F^2\Wiso\bigg(\frac{F}{\sqrt{\det F}}\bigg).(H,H)+D_F^2f(\det F).(H,H)\,.
\end{align}
Beginning with the second derivative of the volumetric part, we first find
\begin{align}
	D_Ff(\det F).H=f'(\det F)\.D_F(\det F).H\,.
\end{align}
Now, since
\begin{align}
	\det(F+H)=\det F\cdot \det (\id+F^{-1}H)=\det F[1+\tr(F^{-1}H)+\det(F^{-1}H)]=\det F+\det F\cdot \tr(F^{-1}H)+\det H\,,
\end{align}
we obtain
\begin{align}
		D_F(\det F).H=\det F\.\tr(F^{-1}H)=\det F\.\langle F^{-1}H,\id\rangle
		\qquad\text{and}\qquad
		D_F^2(\det F).(H,H)=2\.\det H\,.
\end{align}
The first derivative can equivalently be expressed as
\begin{align}
	D_F(\det F).H= -\langle F^{T},H\rangle+(\tr F)\.(\tr H)\,,
\end{align}
since the Cayley-Hamilton Theorem yields
\begin{align}
	F-\tr F\cdot \id+\det F\cdot F^{-1}=0 \qquad \implies\qquad \det F\cdot F^{-1}=-F+(\tr F)\.\id\,.
\end{align}
%
For the second derivative of the volumetric part, we therefore find
\begin{align}
	D_F^2f(\det F).(H,H)&=f''(\det F)\.[D_F(\det F).H]^2+f' \.D_F^2(\det F).(H,H)\notag\\
		&=f''(\det F)\.[\det F\.\langle F^{-1}H,\id\rangle]^2+2\.f'(\det F) \.\det H\,.
\end{align}
For a rank-one matrix $H=\xi\otimes \eta$, using that $\det(\xi\otimes \eta)=0$, we further obtain
\begin{align}
	D_F^2f(\det F)).(\xi\otimes \eta,\xi\otimes \eta)&=f''(\det F)\.[\det F\.\langle F^{-1}\xi\otimes \eta,\id\rangle]^2=f''(\det F)\.(\det F)^2\.[\.\langle F^{-1}\xi ,\eta\rangle]^2.
\end{align}

\bigskip\noindent
Now, we turn our attention to the derivatives of the isochoric part, finding
\begin{align}
	D_F\Wiso\bigg(\frac{F}{\sqrt{\det F}}\bigg).H=\langle D\.\Wiso\bigg(\frac{F}{\sqrt{\det F}}\bigg), D_F\bigg(\frac{F}{\sqrt{\det F}}\bigg).H\rangle\,,
\end{align}
and
\begin{align}
	D_F^2\Wiso\bigg(\frac{F}{\sqrt{\det F}}\bigg).(H,H)&=D^2\.\Wiso\bigg(\frac{F}{\sqrt{\det F}}\bigg). \Big[D_F\bigg(\frac{F}{\sqrt{\det F}}\bigg).H,D_F\bigg(\frac{F}{\sqrt{\det F}}\bigg).H\Big]\notag\\
	&\phantom{=}\;\;+\langle D\.\Wiso\bigg(\frac{F}{\sqrt{\det F}}\bigg), D_F^2\bigg(\frac{F}{\sqrt{\det F}}\bigg).(H,H)\rangle\,.
\end{align}
We compute
\begin{align}
	D_F\bigg(\frac{F}{\sqrt{\det F}}\bigg).H=D_F(F\.(\det F)^{-1/2}).H&=H\.(\det F)^{-1/2}-F\.\frac{1}{2}\.(\det F)^{-3/2}\.\det F\.\langle F^{-T},H\rangle\notag\\
	&=\frac{H}{(\det F)^{1/2}}-\frac{1}{2}\.\frac{F}{(\det F)^{1/2}}\.\langle F^{-T},H\rangle\,.
\end{align}
as well as
\begin{align}
	D_F^2&\bigg(\frac{F}{\sqrt{\det F}}\bigg).(H,H)\notag\\
	&=- \frac{1}{2}\.H\.(\det F)^{-3/2}\.\det F\.\langle F^{-T},H\rangle-\frac{1}{2}\.\frac{H}{(\det F)^{1/2}}\.\langle F^{-T},H\rangle+\frac{3}{4}\.{F}(\det F)^{-5/2}\.[\det F\.\langle F^{-T},H\rangle]^2\notag\\
	&\phantom{=}\;\;-\frac{1}{2}\frac{F}{(\det F)^{3/2}} \langle -H^T+\tr H\cdot \id, H\rangle\\
	&=-\frac{H}{(\det F)^{1/2}},\langle F^{-T},H\rangle+\frac{3}{4}\.{F}(\det F)^{-5/2}\.[\det F\.\langle F^{-T},H\rangle]^2-\frac{1}{2}\frac{F}{(\det F)^{3/2}} \langle -H^T+\tr H\cdot \id, H\rangle\,,\notag
\end{align}
since
\begin{align}
	F^2-\tr F\cdot F+\det F\cdot \id=0
	\;\;&\implies\;\; F-\tr F\cdot \id+\det F\cdot F^{-1}=0\notag\\
	\;\;&\implies\;\; F^T-\tr F\cdot \id+\det F\cdot F^{-T}=0
	\;\;\implies\;\; \det F\cdot F^{-T}= -F^T+\tr F\cdot \id
\end{align}
by virtue of the Cayley-Hamilton Theorem and
\begin{align}
	D_F (\det F\cdot F^{-T}).H=D_F(-F^T+\tr F\cdot \id)=-H^T+\tr H\cdot \id\,.
\end{align}
Moreover,
\begin{align}
	\langle -H^T+\tr H\cdot \id, H\rangle=-\tr(H^2)+(\tr H)^2=2\.\det H
\end{align}
and thus
\begin{align}
	D_F^2\bigg(\frac{F}{\sqrt{\det F}}\bigg).(H,H)=-\.\frac{H}{(\det F)^{1/2}}\.\langle F^{-T},H\rangle+\frac{3}{4}\.\frac{F}{(\det F)^{3/2}}\.[\langle F^{-T},H\rangle]^2-\frac{1}{2}\frac{F}{(\det F)^{3/2}} 2\.\det H\,.
\end{align}
Assuming again that $H=\xi\otimes \eta$ is a rank-one matrix with $\det(\xi\otimes\eta)=0$, we find
\begin{align}
	D_F\bigg(\frac{F}{\sqrt{\det F}}\bigg).(\xi\otimes \eta)=\frac{\xi\otimes \eta}{(\det F)^{1/2}}-\frac{1}{2}\.\frac{F}{(\det F)^{1/2}}\.\langle F^{-T},\xi\otimes \eta\rangle
\end{align}
and
\begin{align}
	D_F^2\bigg(\frac{F}{\sqrt{\det F}}\bigg).(\xi\otimes \eta,\xi\otimes \eta)=-{\xi\otimes \eta}\.\frac{1}{(\det F)^{1/2}}\.\langle F^{-T},\xi\otimes \eta\rangle+\frac{3}{4}\.\frac{F}{(\det F)^{1/2}}\.[\langle F^{-T},\xi\otimes \eta\rangle]^2.
\end{align}
Thus, we have deduced
\begin{align}
	D_F^2&\Wiso\bigg(\frac{F}{\sqrt{\det F}}\bigg).(\xi\otimes \eta,\xi\otimes \eta)\\
	&=D^2\.\Wiso\bigg(\frac{F}{\sqrt{\det F}}\bigg). \Big[\frac{\xi\otimes \eta}{(\det F)^{1/2}}-\frac{1}{2}\.\frac{F}{(\det F)^{1/2}}\.\langle F^{-T},\xi\otimes \eta\rangle,\frac{\xi\otimes \eta}{(\det F)^{1/2}}-\frac{1}{2}\.\frac{F}{(\det F)^{1/2}}\.\langle F^{-T},\xi\otimes \eta\rangle\Big]\notag\\
	&\phantom{=}\;+\langle D\.\Wiso\bigg(\frac{F}{\sqrt{\det F}}\bigg),-{\xi\otimes \eta}\.\frac{1}{(\det F)^{1/2}}\.\langle F^{-T},\xi\otimes \eta\rangle+\frac{3}{4}\.\frac{F}{(\det F)^{1/2}}\.[\langle F^{-T},\xi\otimes \eta\rangle]^2\rangle\,.\notag
\end{align}
In conclusion
\begin{align}
	&\!D_F^2W(F).(\xi\otimes \eta,\xi\otimes \eta)\\
	&=D^2\.\Wiso\bigg(\frac{F}{\sqrt{\det F}}\bigg). \Big[\frac{\xi\otimes \eta}{(\det F)^{1/2}}-\frac{1}{2}\.\frac{F}{(\det F)^{1/2}}\.\langle F^{-T},\xi\otimes \eta\rangle,\frac{\xi\otimes \eta}{(\det F)^{1/2}}-\frac{1}{2}\.\frac{F}{(\det F)^{1/2}}\.\langle F^{-T},\xi\otimes \eta\rangle\Big]\notag\\
	&\phantom{=}\;+\langle D\.\Wiso\bigg(\frac{F}{\sqrt{\det F}}\bigg),-{\xi\otimes \eta}\.\frac{1}{(\det F)^{1/2}}\.\langle F^{-T},\xi\otimes \eta\rangle+\frac{3}{4}\.\frac{F}{(\det F)^{1/2}}\.[\langle F^{-T},\xi\otimes \eta\rangle]^2\rangle+f''(\det F)\.(\det F)^2\.[\.\langle F^{-T},\xi\otimes \eta\rangle]^2\notag\\[0.5em]
	&=D^2\.\Wiso\bigg(\frac{F}{\sqrt{\det F}}\bigg). \Big[\frac{\xi\otimes \eta}{(\det F)^{1/2}}-\frac{1}{2}\.\frac{F}{(\det F)^{1/2}}\.\langle F^{-1}\xi, \eta\rangle,\frac{\xi\otimes \eta}{(\det F)^{1/2}}-\frac{1}{2}\.\frac{F}{(\det F)^{1/2}}\.\langle F^{-1}\xi, \eta\rangle\Big]\notag\\
	&\phantom{=}\;\;+\langle D\.\Wiso\bigg(\frac{F}{\sqrt{\det F}}\bigg), {\xi\otimes \eta}\.\frac{1}{(\det F)^{1/2}}\.\langle F^{-1}\xi\,,\; \eta\rangle+\frac{3}{4}\.\frac{F}{(\det F)^{1/2}}\.[\langle F^{-1}\xi, \eta\rangle]^2\rangle+f''(\det F)\.(\det F)^2\.[\.\langle F^{-1}\xi, \eta\rangle]^2.\notag
\end{align}

\bigskip\noindent
Finally, we consider the common representation of the isochoric part in terms of the distortion function $\K\col\GLp(2)\to[1,\infty)$, cf.\ \eqref{eq:muKenergy}. More specifically, let $\psi\col[1,\infty)\to\R$ denote the uniquely determined function with
\begin{align}
	\Wiso\bigg(\frac{F}{\sqrt{\det F}}\bigg)
	= \psi(\K(F))
	=\psi\bigg(\frac{1}{2}\.\frac{\|F\|^2}{\det F}\bigg)
\end{align}
for all $F\in\GLp(2)$. Then
\begin{align}
	D\.\Wiso(X).H=\psi'(\frac{1}{2}\.\|X\|^2)\.\langle X,H\rangle
	\,,\qquad
	D^2\.\Wiso(X).(H,H)=\psi''(\frac{1}{2}\.\|X\|^2)\.[\langle X,H\rangle]^2+\psi'(\frac{1}{2}\.\|X\|^2)\.\| H\|^2
\end{align}
and thus
\begin{align}
	&\!D_F^2W(F).(\xi\otimes \eta,\xi\otimes \eta)\\
	&=\psi''\bigg(\frac{1}{2}\.\frac{\|F\|^2}{\det F}\bigg)\.[\langle \frac{F}{(\det F)^{1/2}},\frac{\xi\otimes \eta}{(\det F)^{1/2}}-\frac{1}{2}\.\frac{F}{(\det F)^{1/2}}\.\langle F^{-1}\xi, \eta\rangle\rangle]^2+\psi'\bigg(\frac{1}{2}\.\frac{\|F\|^2}{\det F}\bigg)\.\| \frac{\xi\otimes \eta}{(\det F)^{1/2}}-\frac{1}{2}\.\frac{F}{(\det F)^{1/2}}\.\langle F^{-1}\xi, \eta\rangle\|^2\notag\\
	&\phantom{=}\;+\psi'\bigg(\frac{1}{2}\.\frac{\|F\|^2}{\det F}\bigg)\.\langle \frac{F}{(\det F)^{1/2}}, -{\xi\otimes \eta}\.\frac{1}{(\det F)^{1/2}}\.\langle F^{-1}\xi, \eta\rangle+\frac{3}{4}\.\frac{F}{(\det F)^{1/2}}\.[\langle F^{-1}\xi, \eta\rangle]^2\rangle+f''(\det F)\.(\det F)^2\.[\.\langle F^{-1}\xi, \eta\rangle]^2\notag\\[0.5em]
	&=\psi''\bigg(\frac{1}{2}\.\frac{\|F\|^2}{\det F}\bigg)\.\frac{1}{(\det F)^2} \.[\langle F,{\xi\otimes \eta}\rangle -\frac{1}{2}\.\|F\|^2\langle F^{-1}\xi, \eta\rangle]^2\\
	&\phantom{=}\;+\psi'\bigg(\frac{1}{2}\.\frac{\|F\|^2}{\det F}\bigg)\.\frac{1}{\det F} \.\left[\| \xi\otimes \eta\|^2-\langle F,{\xi\otimes \eta}\rangle\.\langle F^{-1}\xi, \eta\rangle+\frac{1}{4}\.\|F\|^2\.[\langle F^{-1}\xi, \eta\rangle]^2\right]\notag\\
	&\phantom{=}\;+\psi'\bigg(\frac{1}{2}\.\frac{\|F\|^2}{\det F}\bigg)\.\frac{1}{\det F} \.\left[-\langle F, {\xi\otimes \eta}\rangle\.\langle F^{-1}\xi, \eta\rangle+\frac{3}{4}\.\|F\|^2\.[\langle F^{-1}\xi, \eta\rangle]^2 \right]+f''(\det F)\.(\det F)^2\.[\.\langle F^{-1}\xi, \eta\rangle]^2\notag\\[0.5em]
	&=\psi''\bigg(\frac{1}{2}\.\frac{\|F\|^2}{\det F}\bigg)\.\frac{1}{(\det F)^2} \.[\langle F,{\xi\otimes \eta}\rangle -\frac{1}{2}\.\|F\|^2\langle F^{-T},{\xi\otimes \eta}\rangle]^2\\
	&\phantom{=}\;+\psi'\bigg(\frac{1}{2}\.\frac{\|F\|^2}{\det F}\bigg)\.\frac{1}{\det F} \.\left[\| \xi\otimes \eta\|^2-2\.\langle F,{\xi\otimes \eta}\rangle\.\langle F^{-T},{\xi\otimes \eta}\rangle+\|F\|^2\.[\langle F^{-T},{\xi\otimes \eta}\rangle]^2 \right]+f''(\det F)\.(\det F)^2\.[\.\langle F^{-T},{\xi\otimes \eta}\rangle]^2\notag
	.
\end{align}
At this point, it seems unfeasible to continue with the approach outlined above. In particular, methods based on direct computations in terms of matrix-valued arguments instead of a singular value based representation appear to be highly unsuited for drawing any further conclusions regarding the rank-one convexity of $\Wiso$ and $\Wvol$, respectively.
\end{appendix}
\end{document}